\documentclass[12pt]{amsart}
\usepackage{latexsym,amssymb,amsmath}
\usepackage{amsmath,amsfonts}
\usepackage{epsfig}
\usepackage{graphicx}
\usepackage{verbatim}
\newtheorem{theorem}{Theorem}[section]
\newtheorem{lemma}[theorem]{Lemma}
\newtheorem{proposition}[theorem]{Proposition}

\newtheorem{cor}[theorem]{Corollary}
\newtheorem{definition}[theorem]{Definition}
\newtheorem{rmk}[theorem]{Remark}

\allowdisplaybreaks 

\def\R{{\mathbb R}}

\def\mG{{\mathcal G}}
\def\la{\langle}
\def\ra{\rangle}

\def\uu{{\underline{u}}}

\def\uu{\"u}

\def\mF{{\mathcal F}}

\def\les{\lesssim}
\def\1{{\bf 1}}
\def\oo{{\"o}}

\def\eqnn{\begin{eqnarray*}}
\def\eeqnn{\end{eqnarray*}}
\def\eqn{\begin{eqnarray}}
\def\eeqn{\end{eqnarray}}
\newcommand{\nc}{\newcommand}
\nc{\be}{\begin{equation}}
\nc{\ee}{\end{equation}}
\nc{\ba}{\begin{eqnarray}}
\nc{\ea}{\end{eqnarray}}
\nc{\eps}{\epsilon}
\def\prf{\begin{proof}}
\def\endprf{\end{proof}}
\begin{document}

\title[Derivative NLS on the half line]{The derivative nonlinear Schr\"odinger equation on the half line}

\author[ Erdo\u{g}an, G{\uu}rel, Tzirakis]{M. B. Erdo\u{g}an, T. B. G{\uu}rel, and N. Tzirakis}
\thanks{The first author is partially supported by NSF grant  DMS-1501041. The second author is supported by a grant from the Fulbright Foundation, and thanks the University of Illinois for its hospitality. The third  author's work was supported by a grant from the Simons Foundation (\#355523 Nikolaos Tzirakis)}
 
\address{Department of Mathematics \\
University of Illinois \\
Urbana, IL 61801, U.S.A.}
\email{berdogan@illinois.edu}

\address{Department of Mathematics \\
University of Illinois \\
Urbana, IL 61801, U.S.A.}
\curraddr{Department of Mathematics, Bo\u gazi\c ci University, Bebek 34342, Istanbul, Turkey}
\email{bgurel@boun.edu.tr}

\address{Department of Mathematics \\
University of Illinois \\
Urbana, IL 61801, U.S.A.}
\email{tzirakis@illinois.edu}

\date{}

\begin{abstract}
We study the initial-boundary value problem for the derivative nonlinear Schr\"odinger (DNLS) equation. More precisely we study the wellposedness theory and the regularity properties of the DNLS equation on the half line. We prove almost sharp local wellposedness, nonlinear smoothing, and small data global wellposedness in the energy space.  One of the obstructions is that the crucial  gauge transformation we use replaces the boundary condition with a nonlocal one. We resolve this issue by running an additional  fixed  point argument.
 Our method also implies almost sharp  local and  small energy global wellposedness,  and an improved smoothing estimate for the quintic Schr\"odinger equation on the half line. In the last part of the paper we consider the DNLS equation on $\R$ and prove smoothing estimates by combining the restricted norm method with a normal form transformation.
\end{abstract}

\maketitle
\section{Introduction}

The main purpose of this paper is to study various aspects of the  derivative nonlinear Schr\"odinger (DNLS)  equation as an initial-boundary value problem posed on the half line. The Cauchy problem on $\R$ associated with this equation,  
\be \label{DNLS}\left\{\begin{array}{l}
iq_t+q_{xx}-i (|q|^2q)_x =0,\,\,\,\,x\in\R , t\in \R,\\
q(x,0)=G(x),  
\end{array}\right. 
\ee
describes a variety of physical phenomena and has been extensively studied in the last 20-30 years.  The model \eqref{DNLS} was derived in \cite{mom} and \cite{mj} as a model for the propagation of circularly polarized Alfv\'en waves in magnetized plasma with a constant magnetic field. It also arises in the study of wave propagation in optical fibers \cite{GP}. The equation appears in many other contexts and for more information the reader can consult \cite{ss}, \cite{cyl}, and the references therein. 

We concentrate on the initial-boundary value problem with a Dirichlet boundary condition  on the semi-infinite interval $(0,\infty)$:
\be \label{DNLSR+}\left\{\begin{array}{l}
iq_t+q_{xx}-i (|q|^2q)_x =0,\,\,\,\,x\in\R^+, t\in \R^+,\\
q(x,0)=G(x), \,\,\,\,q(0,t)=H(t).   
\end{array}\right.
\ee
Here  $G\in H^s(\R^+)$ and $H\in H^{\frac{2s+1}{4}}(\R^+)$, with the additional compatibility condition $G(0)=H(0)$ for $s>\frac12$. The compatibility condition is necessary since the solutions we are interested in are    continuous space-time functions for $s>\frac12$.  
This problem bears  significant physical importance as described in  \cite{cyl}:
``Solutions of the DNLS equation under both the vanishing boundary conditions (VBC) and the nonvanishing boundary conditions (NVBC) are physically interesting topics. For problems of nonlinear Alfv\'en waves, weak nonlinear electromagnetic waves in magnetic and dielectric media, waves propagating strictly parallel to the ambient magnetic fields are modelled by the DNLS equation with VBC while those oblique waves are modelled by the DNLS equation  with NVBC. For problems in optical fibres, pulses under bright background waves should be modelled by NVBC."

 The  DNLS equation on $\R$ is known, \cite{ht1}, to be locally wellposed in $H^{s}$ for any $s\geq \frac12$. This result is sharp since for $s<\frac12$ the data to solution map fails to be uniformly continuous, see \cite{ht1} and \cite{bl} for the detailed argument. For earlier partial results on smoother spaces, see \cite{tsufu, tsufu1, ho, nh, ho1, to}. Global-in-time existence and uniqueness in the energy space $H^1$ was proved in \cite{to} assuming the smallness condition $\|G\|_{L^2}<\sqrt{2\pi} $. This result was improved in \cite{ht} to obtain global wellposedness below the energy space using the high low frequency decomposition method of Bourgain \cite{Bbook}. Global wellposedness for any $s>\frac12$ was obtained in \cite{ckstt, ckstt1} using the almost conservation law method also known as the $I$-method. The endpoint,  $s=\frac12$, global theory was established in \cite{miawu}. All the above results have the same smallness assumption which was needed because the energy functional is not positive definite. The smallness condition was later weakened to $\|G\|_{L^2}<\sqrt{4\pi} $ for energy data, see \cite{yw}  and \cite{yw1}. For the  analogous result with  initial data in $H^s$, $s\geq \frac12$, see \cite{guo}. The optimality of the constant $\sqrt{4\pi}$ in the smallness condition is an open problem, in particular there are no known blowup solutions even for negative energy.  
 For global wellposedness in weighted Sobolev spaces with the smallness assumption replaced by   certain spectral assumptions one can  consult \cite{lps1,lps2}.  

The equation on the real line is completely integrable and it has infinitely many conservation laws.  It thus can be analyzed by the inverse scattering transform,  see e.g. \cite{kaup,lee,kondo,lps1,lps2}. For initial-boundary value problems a variant of the inverse scattering method has been developed in \cite{fokas} and applied to many dispersive equations.  In particular, for the smooth solutions of the DNLS equation on the half line see \cite{jl} and \cite{jl1}. Our work in this paper does not rely on the integrability structure.

We study  the DNLS equation on the half line by  an adaptation of the restricted norm method ($X^{s,b}$ method) to the initial-boundary value problems developed in \cite{ETNLS, et2}. The method is based on ideas that were applied earlier to dispersive equations on the half line and especially for the NLS and the KdV equations in low regularity spaces in the papers \cite{collianderkenig,holmer1,holmer,BSZ,bonaetal}. A well known problem in the theory of the DNLS equation is that the bilinear $X^{s,b}$ estimates fail for the equation \eqref{DNLS}. Nevertheless one can use a gauge transformation that replaces the problematic term $(|q|^2q)_{x}$  with a quintic term which contains no derivatives and a derivative term that has a better convolution structure on the Fourier side, see equation \eqref{first}. In this paper we follow the same approach and thus define the gauge transformation\footnote{We note that this gauge transformation  is slightly different than the one that is commonly used in the literature; this version is more suitable for the boundary value problem.} 
$$\mG_\alpha f(x)= f(x)\exp\Big( -i\alpha\int_x^\infty |f(y)|^2dy\Big), \,\,\,\,\,\,\alpha\in\R.$$  If $q$ solves \eqref{DNLSR+}, then $ u=\mG_\alpha q$ satisfies
\be \label{DNLS_Ga}\left\{\begin{array}{l}
iu_t+u_{xx}-i (2\alpha+1)u^2\overline{u}_x -i(2\alpha+2) |u|^2u_x+\frac{\alpha}2(2\alpha+1)|u|^4u=0,\,\,\,\,x  , t\in \R^+ ,\\
u(x,0) =g(x),\,\,\,\,\,\,u(0,t) = h(t), 
\end{array}\right.
\ee
where $g(x)= \mG_\alpha G(x),$ and
$$  h(t) = H(t) \exp\Big( -i\alpha\int_0^\infty |q(y,t)|^2dy\Big)=H(t) \exp\Big( -i\alpha\int_0^\infty |u(y,t)|^2dy\Big). $$ The second equality holds since the gauge  transformation is unitary. The  equation \eqref{DNLS_Ga} has a counterpart on $\R$, which  is identical without the boundary value $h$. 

We will first establish the local wellposedness of \eqref{DNLS_Ga} in the case $\alpha=-1$. The equation then becomes
\be \label{first}\left\{
\begin{array}{l}
iu_t+u_{xx}+iu^2\overline{u}_x +\frac12 |u|^4u=0,\,\,\,\,x  , t\in \R^+ ,\\
u(x,0) =g(x),\,\,\,\,\,\,u(0,t) = h(t).
\end{array} \right.
\ee

The second step will be to establish the theory for any other $\alpha$. This is trivial in the case of the real line case since the gauge transformation is a bi-Lipschitz map between Sobolev spaces, see \cite{ckstt} and Lemma~\ref{lem:contgauge} in the Appendix. However, because of the dependence of the boundary data on the $L^2$ norm of  the solution in the gauged equation, the local theory on the half line for any other $\alpha$ requires an additional fixed point argument that allows us to pick the boundary data for $\alpha=-1$ problem suitably, see Section~\ref{sec:genalpha}. 

Wellposedness of \eqref{DNLS_Ga} or \eqref{first} means local existence, uniqueness and continuity with respect to the initial data of distributional solutions. More precisely we have the following definition: 

\begin{definition}\label{def:lwp}
We say \eqref{DNLS_Ga} is locally wellposed in $H^s(\R^+)$, if \\
i) for any $g\in H^s(\R^+)$ and $h\in H^{\frac{2s+1}{4}}(\R^+)$, with the compatibility condition $g(0)=h(0)$, the equation has a distributional solution 
$$
u\in   C^0_tH^s_x([0,T]\times \R) \cap C^0_xH^{\frac{2s+1}{4}}_t(\R \times [0,T]),
$$
where $T=T(\|g\|_{H^s(\R^+)},\|h\|_{H^{\frac{2s+1}4(\R^+)}})$, \\
ii) if $g_n\to g$ in $H^s(\R^+)$ and $h_n\to h$ in $H^{\frac{2s+1}{4}}(\R^+)$, then $u_n\to u$ in the space above,\\
iii)   $u$ is the unique such solution.
\end{definition}

Our first theorem establishes almost sharp local wellposedness (up to the endpoint $s=\frac12$).

\begin{theorem} \label{thm:local} Fix  $s\in (\frac12, \frac52)$, $s\neq \frac32$.  Then for any fixed $\alpha\in\R$, \eqref{DNLS_Ga} is locally wellposed in $H^s(\R^+)$ in the sense of Definition~\ref{def:lwp}.
\end{theorem}

As mentioned above we start with the case $\alpha=-1$, c.f. equation \eqref{first}. In this particular case we also obtain a smoothing estimate:
 \begin{theorem} \label{thm:smooth} Fix $s\in (\frac12, \frac52)$, $s\neq \frac32$, and $a<\min(\frac52-s, \frac14,2s-1)$.  Then  for any $g\in H^s(\R^+)$ and $h\in H^{\frac{2s+1}{4}}(\R^+)$, with the compatibility condition $g(0)=h(0)$,   the solution $u$ of \eqref{first}  satisfies
 $$
 u(x,t)-W_0^t(g,h)(x)\in C^0_tH^{s+a}_x([0,T]\times \R^+),
 $$
 where $T$ is the local existence time, and  $ W_0^t(g,h)$ is the solution of the corresponding linear equation.
\end{theorem}
Notice that Theorem \ref{thm:smooth} explicitly states  that the nonlinear part of the solution is smoother than the initial data and the corresponding linear solution. This smoothing estimate plays a central role in the uniqueness part of Theorem \ref{thm:local} for $\alpha =-1$ and thus for any $\alpha$,  since it shows that if we approximate any solution with a smoother solution (uniqueness of which is known by energy methods) the time interval on which this approximation is valid depends only on the norm of the less regular solutions. This argument, see \cite{CT}, provides uniqueness for rough solutions at any regularity level that the solutions exist. For details see Section 4.

\begin{rmk}
The a priori estimates we need to obtain in order to prove Theorem \ref{thm:smooth}, in particular Proposition \ref{prop:smooth} and Proposition \ref{prop:smooth2}, also imply almost sharp local wellposedness and smoothing of order $a<\min(4s, \frac{1}{2}, \frac52-s)$ for any $s>0$ for the $L^2$-critical quintic NLS equation $iu_t+u_{xx}\pm |u|^4u=0$
on the half line, which as far as we know is a new result by itself. 
In addition, these estimates imply a new smoothing estimate for the $L^2$-critical quintic NLS equation on $\Bbb R$ of order $a<\min(4s, \frac{1}{2})$  improving earlier smoothing estimates obtained in \cite{kervar} and \cite{egt}.
\end{rmk}

We next establish the global wellposedness in the energy space for the equation \eqref{DNLS_Ga} for any $\alpha$, in particular for the equation \eqref{DNLSR+}. For boundary value problems with nonzero boundary data this is not an immediate consequence of the local theory since the local differentiation laws do not always lead to conservation laws. Nevertheless, in order to extend the solutions to all times we only need an a priori bound of the $H^1$ norm of the solution. On the half line this is indeed the case as we prove the bound for small initial and boundary data in the energy space. The proof is done in two steps. First we obtain the bounds in the case that $\alpha=-\frac12$ where the differentiation laws take the simplest form, and then we prove similar bounds for the DNLS equation \eqref{DNLS_Ga} for any $\alpha$ by substituting the gauge in the local energy identities:
\begin{theorem}\label{thm:global}
For any $\alpha\in\R$, there exists an absolute constant $c>0$ so that \eqref{DNLS_Ga} is globally wellposed in $H^1(\R^+)$ provided that $\|g\|_{H^1(\R^+)}+\|h\|_{H^1(\R^+)}\leq c$.
\end{theorem}

In the last part of the paper we consider the derivative Schr\"odinger equation on the full line \eqref{DNLS}. We combine the $ X^{s,b}$ theory with the theory of normal forms as was developed
in \cite{bit} and \cite{et1} for the periodic KdV equation to obtain the following smoothing theorem:

\begin{theorem}\label{thm:smoothR} Fix  $s>\frac12$ and $a<\min(\frac12,2s-1)$. 
For any $g\in H^s(\R)$ the solution $u $ of 
\be \label{DNLS-1}\left\{\begin{array}{l}
iu_t+u_{xx}+i  u^2\overline{u}_x +\frac{1}2 |u|^4u=0,\,\,\,\,x  , t\in \R,\\
u(x,0) =g(x) 
\end{array}\right.
\ee
satisfies
$$
u-e^{it\Delta}g\in C^0_tH^{s+a}_x([0,T]\times \R),
$$
where $T$ is the local existence time.
\end{theorem}
This result can be iterated to hold for all times, see \cite{et1}, under the small $L^2$ assumption that guarantees global existence. Since the gauge  transformation is continuous in Sobolev spaces, see Lemma~\ref{lem:contgauge}
in the Appendix,  
this theorem immediately implies that 
\begin{cor} Fix  $s>\frac12$ and $a<\min(\frac12,2s-1)$. 
For any $G\in H^s(\R)$ the solution   $q$ of \eqref{DNLS} satisfies
\be \label{Rsmooth}
q-\mG_1\big(e^{it\Delta}(\mG_{-1}G)\big)\in C^0_tH^{s+a}_x([0,T]\times \R). 
\ee
\end{cor}
This improves a smoothing result\footnote{Although it is not stated explicitly in \cite{ht}, the assertion \eqref{Rsmooth} follows from the a priori estimates obtained there for the smaller range $a<\min(\frac14,2s-1)$. } that was obtained in \cite{ht} for equation \eqref{DNLS}.

We now briefly discuss the organization of the paper.   In Section~\ref{sec:defin} we define the notion of a solution, discuss the solution of the linear equation,  and obtain an integral formulation on the full line that we need in order to run a fixed point argument, see equation \eqref{eq:duhamel}.  We thus are looking for a fixed point in the space 
\be\label{eq:space}
X^{s,b}(\R\times [0,T]) \cap C^0_tH^s_x([0,T]\times \R) \cap C^0_xH^{\frac{2s+1}{4}}_t(\R \times [0,T]).
\ee

It is a well known fact that  (see \eqref{def:xsb} below for the definition of the $X^{s,b}$ norm)
$X^{s,b}(\R\times [0,T]) \subset C^0_tH^s_x([0,T]\times \R)$
for any $b>\frac{1}{2}$. However, to close the fixed point argument we need to take  $b<\frac{1}{2}$ and prove the continuity of the solution directly by additional estimates. In Section~\ref{sec:apriori}  we prove the linear and nonlinear a priori estimates that  are useful in studying  wellposedness  of the derivative NLS on the half line.    In Section~\ref{sec:prt3}  we establish the local wellposedness theory for general $\alpha$, see Theorem~\ref{thm:local}.  
In Section~\ref{sec:global} we obtain the global wellposedness  with small mass and energy by proving a priori bounds on the energy norm.  Section~\ref {sec:DNLSR} is devoted to the derivative NLS equation on the real line.    In particular, we apply a normal form transformation and prove multilinear estimates that we use to obtain an improved smoothing bound for the equation.  Finally in the Appendix, we record a lemma that we use repeatedly throughout the paper and prove another lemma on the Lipschitz continuity of   the gauge transformation  in Sobolev spaces.

\subsection{Notation}
We define the Fourier transform on $\R$ by
$$
\widehat g(\xi)=\mathcal F g(\xi)=\int_{\R } e^{-ix \xi} g(x) dx.
$$ 
We also define the Sobolev space $H^s(\R)$ via the norm:
$$
\|g\|_{H^s}=\|g\|_{H^s(\R)}=\Big(\int_\R \la \xi\ra^{2s} |\widehat g(\xi)|^2 d\xi\Big)^{1/2},
$$
where $\la \xi\ra:=\sqrt{1+|\xi|^2}$. 
We denote the linear Schr\"odinger propagator (for $g\in L^2(\R)$) by
$$
W_\R g(x,t)=e^{it\Delta} g(x)= \mathcal F^{-1}\big[e^{-it|\cdot|^2} \widehat g(\cdot)\big](x).
$$
For a space time function $f$, we set the notation
$$
D_0f(t)=f(0,t).
$$
Finally, we  reserve  the symbol $\eta $ for a smooth compactly supported function of time which is equal to $1$ on $[-1,1]$. 

\section{Notion of a solution} \label{sec:defin}

Throughout the paper we have $s\in(0,\frac52) $, $s\neq\frac12, \frac32$.
We define $H^s(\R^+)$ norm as
$$
\|g\|_{H^s(\R^+)}:=\inf\big\{\|\tilde g\|_{H^s(\R)}: \tilde g(x)=g(x),\, x>0\big\}.
$$
We say $\tilde g$ is an $H^s(\R)$ extension of $g\in H^s(\R^+)$ if $ \tilde g(x)=g(x)$ for $x>0$ and $\|\tilde g\|_{H^s}\leq 2 \|g\|_{H^s(\R^+)}$.  
Note that, if $g  \in  H^s(\R^+)$ for some $s>\frac12$, then by Sobolev embedding  any $H^s$ extension  is continuous on $\R$, and hence $g(0)$ is well defined. 
We have the following lemma concerning extensions of $H^s(\R^+)$ functions, see \cite{collianderkenig} and  \cite{ETNLS}:
\begin{lemma}\label{lem:Hs0} Let $h\in H^s(\R^+)$ for some $-\frac12< s<\frac32$. \\
i) If $-\frac12< s<\frac12$, then $\|\chi_{(0,\infty)}h\|_{H^s(\R)}\les \|h\|_{H^s(\R^+)}$.\\
ii) If $\frac12<s<\frac32$  and $h(0)=0$, then $\|\chi_{(0,\infty)}h\|_{H^s(\R)}\les \|h\|_{H^s(\R^+)}$. 
\end{lemma}
 
To construct the solutions of \eqref{first}   we first consider the linear problem:
\be\label{linearnls} \left\{ \begin{array}{l}
 iu_t+u_{xx} =0,\,\,\,\,x\in\R^+, t\in \R^+,\\
 u(x,0)=g(x)\in H^s(\R^+), \,\,\,\,u(0,t)=h(t)\in H^{\frac{2s+1}4}(\R^+), 
\end{array}
\right.
\ee
with the compatibility condition $h(0)=g(0)$ for $s>\frac12$. Note that the uniqueness of the solutions of equation \eqref{linearnls} follows  by considering the equation with $g=h=0$ with the method of odd extension. 

We refer the reader to \cite{bonaetal} and  \cite{ETNLS} for the derivation of the solution of  \eqref{linearnls},  for $t\in [0,1]$. We write
$$
u(t)=W_0^t(g,h)=W_0^t(0,h-p)+W_\R(t) g_e,
$$
where $g_e$ is an $H^s$ extension of $g$ to $\R$ satisfying $\|g_e\|_{H^s(\R)}\les \|g\|_{H^s(\R^+)}$. Moreover,  
$$p(t)= \eta(t ) D_0 (W_\R g_e)= \eta(t) [W_\R(t) g_e]\big|_{x=0},$$
 which is well defined and is in  $H^{\frac{2s+1}{4}}(\R^+)$ by Lemma~\ref{lem:kato} below.  
In addition, following \cite{bonaetal} and \cite{ETNLS} we  write  $ W_0^t(0, h)=W_{1}h+W_{2}  h $, where
\begin{align}\label{eq:w1}
W_{1}h(x,t)&=\frac1\pi\int_{0}^\infty e^{-i\beta^2t+i\beta x}\beta \widehat  h(-\beta^2) d\beta,\\
W_{2}h(x,t)&=\frac1\pi \int_{0}^\infty e^{i\beta^2 t-\beta x} \beta \widehat h(\beta^2) d\beta.
\end{align}
Here by a slight abuse of notation
\be\label{eq:halffourier}
\widehat h(\xi)=\mathcal F\big(\chi_{(0,\infty)} h\big)(\xi)=\int_0^\infty e^{-i\xi t } h(t) dt.
\ee
By a change of variable  we have
\be\label{eq:psiineq}
\sqrt{\int_0^\infty \langle \beta\rangle^{2s} \big|\beta \widehat h(\pm \beta^2)\big|^2 d\beta } \les \|\chi_{(0,\infty)} h\|_{H^\frac{2s+1}4(\R )} \les \|h\|_{H^\frac{2s+1}4(\R^+)},
\ee
where the last inequality follows from   Lemma~\ref{lem:Hs0} under the compatibility condition $h(0)=0$.  

Note that $W_{1}$ is already well defined for all $x ,t \in \R$ by \eqref{eq:w1}. We also extend $W_{2}$ to all $x$ as in \cite{ETNLS} by
\be\label{eq:w2}
W_{2}h(x,t) =\frac1\pi \int_{0}^\infty e^{i\beta^2 t-\beta x} \rho(\beta x) \beta \widehat h(\beta^2) d\beta,
\ee
where $\rho(x)$ is a smooth function supported on $(-2,\infty)$, and $\rho(x)=1$ for $x>0$. 
Therefore the solution $W_0^t(g,h) $ of \eqref{linearnls} for $t\in [0,1] $ is  now well defined for all $x,t\in\R$, and its restriction to $\R^+\times [0,1]$ is independent of the extension $g_e$.

We now consider the integral equation
\begin{equation}\label{eq:duhamel}
u(t)=\eta(t)W_0^t(g,h)  + \eta(t) \int_0^tW_\R(t- t^\prime)  F(u) \,d t^\prime - \eta(t) W_0^t\big(0, q  \big)(t),
\end{equation}
where
$$
F(u)= \eta(t/T) (iu^2  \overline{u}_x+\frac12|u|^4u)  \text{ and }  q(t)=\eta(t ) D_0\Big(\int_0^tW_\R(t- t^\prime)  F(u)\, d t^\prime \Big).
$$
Here $D_0f(t)=f(0,t)$.  
In what follows we will prove that the  integral equation \eqref{eq:duhamel} has a unique solution in the Banach space \eqref{eq:space} on $\R\times \R$ for some $T<1$. Using the definition of the boundary operator, it is clear that the restriction of $u$ to $\R^+\times [0,T]$ satisfies  \eqref{first}  in the distributional sense. Also note that the smooth solutions of \eqref{eq:duhamel} satisfy \eqref{first}  in the classical sense.

We will work with the space $X^{s,b}(\R\times\R)$ (see \cite{bourgain,Bou2}):
\be\label{def:xsb}
\|u\|_{X^{s,b}}=\big\| \widehat{u}(\tau,\xi)\la \xi\ra^{s} \la \tau+\xi^2\ra^{b} \big\|_{L^2_\tau L^2_\xi}.
\ee 
We  recall the embedding $X^{s,b}\subset C^0_t H^{s}_x$ for $b>\frac{1}{2}$ and the following inequalities from \cite{bourgain,gtv,etbook}. 
For any $s,b$ we have
\begin{equation}\label{eq:xs1}
\|\eta(t)W_\R g\|_{X^{s,b}}\les \|g\|_{H^s}.
\end{equation}
For any $s\in \mathbb R$,  $0\leq b_1<\frac12$, and $0\leq b_2\leq 1-b_1$, we have
\begin{equation}\label{eq:xs2}
\Big\| \eta(t) \int_0^t W_\R(t-t^\prime)  F(t^\prime ) dt^\prime \Big\|_{X^{s,b_2} }\lesssim   \|F\|_{X^{s,-b_1} }.
\end{equation}
Moreover, for $T<1$, and $-\frac12<b_1<b_2<\frac12$, we have
\begin{equation}\label{eq:xs3}
\|\eta(t/T) F \|_{X^{s,b_1}}\les T^{b_2-b_1} \|F\|_{X^{s,b_2}}.
\end{equation}
Finally, recall that, see e.g. \cite[Lemma 2.9]{taobook}, for any translation invariant Banach space $\mathcal B$ of functions on $\R\times \R$, the a priori estimate $\|W_\R g\|_{\mathcal B} \les \|g\|_{H^s}$ implies that
\begin{equation}\label{eq:xs4}
\| u \|_{\mathcal B}\les \|u\|_{X^{s,b}} \,\,\text{ for any } b>\frac12.
\end{equation}

\section{A priori estimates} \label{sec:apriori}
In this section we provide a priori estimates for the linear and   nonlinear terms in \eqref{eq:duhamel}.
\subsection{Estimates for the linear terms}

We start with the following Kato smoothing type estimates which convert space derivatives to time derivatives. These  estimates justify the choice of spaces concerning $g$, $h$ in Definition~\ref{def:lwp}.   

\begin{lemma}\label{lem:kato}(Kato smoothing inequality) Fix $s\geq 0$. For any $g\in H^s(\R)$, we have
$\eta(t) W_\R g\in C^0_xH^{\frac{2s+1}4}_t(\R\times \R)$, and 
$$
\big\|\eta  W_\R g \big\|_{L^\infty_x H^{\frac{2s+1}{4}}_t }\les \|g\|_{H^s(\R)}.
$$
In addition, for $s\geq \frac12$, and for any $g\in H^s(\R)$, we have
$\eta(t) \partial_x W_\R g\in C^0_xH^{\frac{2(s-1)+1}4}_t(\R\times \R)$, and  
$$
\big\|\eta \partial_x W_\R g \big\|_{L^\infty_x H^{\frac{2(s-1)+1}{4}}_t }\les \|g\|_{H^s(\R)}.
$$
\end{lemma}
 \begin{proof}
The first part is the well known Kato smoothing theorem, see e.g.  \cite{ETNLS}. The second part follows from the first part for $s\geq 1$ since $\partial_x$ commutes with $W_\R$. For $s\in[\frac12,1)$, note that
\begin{multline*}
\mathcal F_t \big(  \eta  \partial_xW_\R g \big)(\tau)=i\int \widehat\eta(\tau+\xi^2) e^{ix\xi} \xi \widehat{g}(\xi) d\xi \\ =i\int_{|\xi|<1} \widehat\eta(\tau+\xi^2) e^{ix\xi} \xi \widehat{g}(\xi) d\xi +i \int_{|\xi|\geq 1} \widehat\eta(\tau+\xi^2) e^{ix\xi}\xi \widehat{g}(\xi) d\xi.
\end{multline*}
We estimate the contribution of the first term to  $H^{\frac{2s-1}{4}}_t$ norm by
$$
\int_{|\xi|<1} \big\|\la\tau\ra^{\frac{2s-1}{4}}\widehat\eta(\tau+\xi^2)\big\|_{L^2_\tau} |\widehat{g}(\xi)| d\xi \les \int_{|\xi|<1} |\widehat{g}(\xi)| d\xi
\les \|\widehat g\|_{L^2}\les \|g\|_{H^s}.
$$
By a change of variable, the contribution of the second term is bounded by
$$ \Big\|\int_1^\infty  \la \tau\ra^{\frac{2s-1}{4}} |\widehat\eta(\tau+\rho)|  |\widehat{g}(\pm\sqrt{\rho})| d\rho\Big\|_{L^2_\tau}
\les \Big\|\int_1^\infty  \la \tau+\rho\ra^{\frac{2s-1}{4}}|\widehat\eta(\tau+\rho)| \rho^{\frac{2s-1}{4}} |\widehat{g}(\pm\sqrt{\rho})| d\rho\Big\|_{L^2_\tau} .
$$
By Young's inequality, we estimate this by
$$
\|\la\cdot\ra^{\frac{2s-1}{4}} \widehat{\eta} \|_{L^1} \Big\|\rho^{\frac{2s-1}{4}} \widehat{g}(\pm\sqrt{\rho}) \Big\|_{L^2_{\rho>1}}\les \|g\|_{H^s}.
$$
The continuity statement follows from this and the dominated convergence theorem.
\end{proof}
Lemma~\ref{lem:wbcont} and Proposition~\ref{prop:wbh} below show that the boundary operator belongs to the space \eqref{eq:space}.  
\begin{lemma}\label{lem:wbcont}  Let  $s\geq 0 $. Then for  $h$ satisfying  $\chi_{(0,\infty)}h\in H^{\frac{2s+1}{4}}(\R )$,   we have
$$W_0^t(0, h) \in C^0_tH^s_x(\R\times \R)\,\,\, and \,\,\,\eta(t)W_0^t(0, h) \in C^0_xH^{\frac{2s+1}4}_t(\R\times \R).$$

In addition, 
 for $s\geq \frac12 $  and for  $h$ satisfying  $\chi_{(0,\infty)}h\in H^{\frac{2s+1}{4}}(\R )$,   we have  
 $$\eta(t)\partial_x W_0^t(0, h) \in C^0_xH^{\frac{2(s-1)+1}4}_t(\R\times \R).$$
\end{lemma} 
\begin{proof} For the first part see \cite{ETNLS}.  
To prove that $\eta(t)\partial_x W_2h \in C^0_xH^{\frac{2s-1}4}_t(\R\times \R)$, write
\begin{align*}
W_2h  = \int_\R f(\beta x)   \mathcal F\big(e^{-it \Delta } \psi\big)(\beta) d\beta = \int_\R \frac1x \widehat f(\xi/x)  (e^{-it \Delta } \psi) (\xi) d\xi
=  \int_\R   \widehat f(\xi )  (e^{-it \Delta } \psi) (x \xi) d\xi,
\end{align*}
where $f(x)=e^{- x} \rho(x)$ and $\widehat \psi(\beta)=\beta \widehat h( \beta^2) \chi_{[0,\infty)}(\beta)$. 
We therefore obtain
$$
\partial_x W_2h =  \int_\R   \xi\widehat f(\xi )  (e^{-it \Delta } \psi^\prime) (x \xi) d\xi.
$$
 
The claim follows from  using the  Kato smoothing Lemma~\ref{lem:kato} and dominated convergence theorem noting that $ \xi \widehat f(\xi)  \in L^1$.

Finally, note that $W_1h=W_\R \psi$,
where $ \widehat \psi(\beta)=\beta \widehat h(-\beta^2) \chi_{[0,\infty)}(\beta)$. 
The claim $\eta(t)\partial_x W_1h \in C^0_xH^{\frac{2s-1}4}_t(\R\times \R)$ follows  as above from \eqref{eq:halffourier}, \eqref{eq:psiineq}, the continuity of $W_\R(t)$, and Kato smoothing Lemma~\ref{lem:kato}.
\end{proof}

We also record the following bound from \cite{ETNLS}: 
\begin{proposition} \label{prop:wbh} Let $b\leq \frac12$ and $s\geq 0 $. Then for  $h$ satisfying  $\chi_{(0,\infty)}h\in H^{\frac{2s+1}{4}}(\R )$,   we have
$$\|\eta(t) W_0^t(0,h) \|_{X^{s,b}} \les \|\chi_{(0,\infty)}h\|_{H_t^{\frac{2s+1}4}(\R )}.$$
\end{proposition}

 The following proposition is a Kato smoothing type estimate for the nonlinear Duhamel term:
 \begin{proposition}\label{prop:duhamelkato} Fix $b<\frac12$.  For any smooth compactly supported function $\eta$, we have
\begin{align*}
\Big\|\eta  \int_0^tW_\R(t- t^\prime) F  dt^\prime  \Big\|_{C^0_xH^{\frac{2s+1}{4}}_t(\R\times \R)}\les \left\{ \begin{array}{ll} \|F\|_{X^{s,-b}},&  0\leq s \leq \frac12,   \\
\|F\|_{X^{\frac12,\frac{2s-1-4b}{4}}} +\|F\|_{X^{s,-b}}, &   \frac12 \leq s \leq \frac52. \end{array}
\right.
\end{align*}

In addition, we have
\begin{align*}
\Big\|\eta  \partial_x \int_0^tW_\R(t- t^\prime) F  dt^\prime  \Big\|_{C^0_xH^{\frac{2s-1}{4}}_t(\R\times \R)}\les \left\{ \begin{array}{ll} \|F\|_{X^{s,-b}},& \frac12\leq s \leq \frac32,  \\
\|F\|_{X^{\frac12,\frac{2s-1-4b}{4}}} +\|F\|_{X^{s,-b}}, &   \frac32 \leq s \leq \frac52. \end{array}
\right.
\end{align*}
\end{proposition}
\begin{proof}  For the first part see \cite{ETNLS}.  The second part follows from the first for $s\geq 1$ since $\partial_x$
commutes with $W_\R$. For $\frac12\leq s <1$,  the proof is based on an argument from \cite{collianderkenig}, also see \cite{ETNLS}. 

It suffices to prove the bound above for $\eta D_0\big(\int_0^tW_\R(t- t^\prime) \partial_x F  dt^\prime \big)$ since $X^{s,b}$ norm is independent of space translation. The continuity in $x$ follows from this by the dominated convergence theorem as in the proof of Lemma~\ref{lem:kato}.  Note that
$$
D_0\Big(\int_0^tW_\R(t- t^\prime) \partial_x F  dt^\prime \Big)= i\int_\R\int_0^t e^{-i(t-t^\prime)\xi^2}\xi F(\widehat\xi,t^\prime) dt^\prime d\xi.
$$
Using
$$
F(\widehat\xi,t^\prime)=\int_\R e^{it^\prime\lambda}\widehat F(\xi,\lambda) d\lambda \,\,\,\,\text{ and }\,\,\,
\int_0^t e^{it^\prime(\xi^2+\lambda)}dt^\prime = \frac{e^{it(\xi^2+\lambda)}-1}{i(\lambda+\xi^2)}
$$
we obtain
$$
D_0\Big(\int_0^tW_\R(t- t^\prime) \partial_x F  dt^\prime \Big) = i\int_{\R^2} \frac{e^{it \lambda }-e^{-it\xi^2}}{i(\lambda+\xi^2)} \xi \widehat F(\xi,\lambda) d\xi d\lambda.
$$
Let $\psi$ be a smooth cutoff for $[-1,1]$, and let $\psi^c=1-\psi$. We write
\begin{multline*}
 \eta(t) D_0\Big(\int_0^tW_\R(t- t^\prime)  \partial_x F  dt^\prime \Big)=   \eta(t)  \int_{\R^2} \frac{e^{it \lambda }-e^{-it\xi^2}}{i(\lambda+\xi^2)} \psi(\lambda+\xi^2) \xi \widehat F(\xi,\lambda) d\xi d\lambda \\ +\eta(t) \int_{\R^2} \frac{e^{it \lambda } }{i(\lambda+\xi^2)} \psi^c(\lambda+\xi^2) \xi \widehat F(\xi,\lambda) d\xi d\lambda
-\eta(t) \int_{\R^2} \frac{ e^{-it\xi^2}}{i(\lambda+\xi^2)} \psi^c(\lambda+\xi^2) \xi \widehat F(\xi,\lambda) d\xi d\lambda \\ =:I+II+III.
\end{multline*}
By Taylor expansion, we write
$$
  \frac{e^{it \lambda }-e^{-it\xi^2}}{i(\lambda+\xi^2)} =ie^{it\lambda} \sum_{k=1}^\infty \frac{(-it)^k}{k!} (\lambda+\xi^2)^{k-1},
$$
which leads to
\begin{multline*}
\|I\|_{H^{\frac{2s-1}{4}}(\R)}\les  \sum_{k=1}^\infty  \frac{\| \eta(t)t^k\|_{H^1}}{k!}  \Big\|  \int_{\R^2}  e^{it\lambda} (\lambda+\xi^2)^{k-1}  \psi(\lambda+\xi^2) \xi  \widehat F(\xi,\lambda) d\xi d\lambda\Big\|_{H_t^{\frac{2s-1}{4}}(\R)}\\
\les  \sum_{k=1}^\infty  \frac{1}{(k-1)!}  \Big\| \la \lambda\ra^{\frac{2s-1}{4}} \int_{\R }    (\lambda+\xi^2)^{k-1}  \psi(\lambda+\xi^2) \xi \widehat F(\xi,\lambda) d\xi  \Big\|_{L^2_\lambda}\\
\les  \Big\| \la \lambda\ra^{\frac{2s-1}{4}} \int_{\R }       \psi(\lambda+\xi^2) |\xi| |\widehat F(\xi,\lambda)| d\xi  \Big\|_{L^2_\lambda}.
\end{multline*}
By the Cauchy-Schwarz inequality in $\xi$ we estimate this by
\begin{multline*}
\Big[  \int_{\R }  \la \lambda\ra^{\frac{2s-1}{2}} \Big(  \int_{|\lambda+\xi^2|<1} \la \xi\ra^{2-2s} d\xi\Big)\Big(  \int_{|\lambda+\xi^2|<1}  \la \xi\ra^{2s} |\widehat F(\xi,\lambda)|^2 d\xi\Big) d\lambda \Big]^{1/2}\\
\les \|F\|_{X^{s,-b}} \sup_\lambda   \Big( \la \lambda\ra^{\frac{2s-1}{2}} \int_{|\lambda+\xi^2|<1} \la \xi\ra^{2-2s} d\xi\Big)^{1/2}\les \|F\|_{X^{s,-b}} .
\end{multline*}
 The last inequality follows by a calculation substituting $\rho =\xi^2$.

For the second term  we have
\begin{multline*}
\|II\|_{H^{\frac{2s-1}{4}}(\R)}\les \|\eta\|_{H^1} \Big\| \la \lambda\ra^{\frac{2s-1}{4}}\int_{\R } \frac{1}{ \lambda+\xi^2 } \psi^c(\lambda+\xi^2) \xi \widehat F(\xi,\lambda) d\xi  \Big\|_{L^2_\lambda}\\
\les  \Big\| \la \lambda\ra^{\frac{2s-1}{4}}\int_{\R } \frac{1}{ \la \lambda+\xi^2 \ra } | \xi \widehat F(\xi,\lambda)| d\xi  \Big\|_{L^2_\lambda}.
\end{multline*}
By the Cauchy-Schwarz inequality in $\xi$  we estimate this by
\begin{multline*}
\Big[  \int_{\R }  \la \lambda\ra^{\frac{2s-1}{2}} \Big(  \int \frac{\la \xi\ra^{2-2s}}{\la \lambda+\xi^2 \ra^{2-2b} } d\xi\Big)\Big(  \int \frac{ \la \xi\ra^{2s}}{\la \lambda+\xi^2 \ra^{2b} } |\widehat F(\xi,\lambda)|^2 d\xi\Big) d\lambda \Big]^{1/2}\\
\les \|F\|_{X^{s,-b}} \sup_\lambda   \Big( \la \lambda\ra^{\frac{2s-1}{2}}  \int \frac{\la \xi\ra^{2-2s}}{\la \lambda+\xi^2 \ra^{2-2b} } d\xi\Big)^{1/2} 
\les \|F\|_{X^{s,-b}} .
\end{multline*}
To obtain the last inequality recall that $\frac12 \leq s \leq 1$ and  $b<\frac12$,  and  consider the cases $|\xi|<1$ and $|\xi|\geq 1$ separately. In the former case use $\la \lambda+\xi^2 \ra\approx \la\lambda\ra$  and in the latter case use Lemma~\ref{lem:sums} after the change of variable $\rho=\xi^2$.

To estimate $\|III\|_{H^{\frac{2s-1}{4}}(\R)}$ we break the $\xi$ integral into two pieces $|\xi|\geq 1$ and $|\xi|<1$. We estimate the contribution of the former piece as above (after the change of variable $\rho=\xi^2$):
$$
   \Big\| \la \rho\ra^{\frac{2s-1}{4}}\int_{\R } \frac{1}{ \lambda+\rho } \psi^c(\lambda+\rho) \widehat F(\sqrt{\rho},\lambda) d\lambda  \Big\|_{L^2_{|\rho|\geq 1}}
\les  \Big\| \la \rho\ra^{\frac{2s-1}{4}}\int_{\R } \frac{1}{ \la \lambda+\rho \ra } | \widehat F(\sqrt{\rho},\lambda)| \ d\lambda   \Big\|_{L^2_{|\rho|\geq 1}}.
$$
By the Cauchy-Schwarz inequality in $\lambda$ integral  noting that $b<\frac12$  we bound this by
$$
   \Big[\int_{|\rho|>1}\int_{\R } \frac{ \la \rho\ra^{\frac{2s-1}{2}}}{ \la \lambda+\rho \ra^{2b} } | \widehat F(\sqrt{\rho},\lambda)|^2   d\lambda  d\rho  \Big]^{1/2}\les \|F\|_{X^{s,-b}}.
$$
We estimate the contribution of the latter term by
$$
 \int_{\R^2} \frac{ \|\eta(t)  e^{-it\xi^2}\|_{H^{\frac{2s-1}{4}}}\chi_{[-1,1]}(\xi)}{|\lambda+\xi^2|} \psi^c(\lambda+\xi^2)  |\xi \widehat F(\xi,\lambda)| d\xi d\lambda \les \int_{\R^2} \frac{   \chi_{[-1,1]}(\xi)}{\la\lambda+\xi^2\ra}    |\widehat F(\xi,\lambda)| d\xi d\lambda.
$$
For $b<\frac12$  this is bounded by $\|F\|_{X^{0,-b}}$ by the Cauchy-Schwarz inequality in $\xi$ and $\lambda$ integrals. 
\end{proof}

\subsection{Estimates for the nonlinear terms}
In this section we establish estimates for the nonlinear term  in \eqref{eq:duhamel} in order to close the fixed point argument and to obtain the smoothing theorem. 

We start by recording a priori estimates that follow  from \eqref{eq:xs4} and well known 
dispersive estimates for the linear Schr{\oo}dinger evolution, for details see  e.g.  \cite{kpv,ht1,ht}:
\be\label{ks}
\Big\|\Big(\frac{\la \xi\ra^{\frac12} f(\xi,\tau) }{\la\tau+\xi^2\ra^{\frac12+}}\Big)^\vee \Big\|_{L^\infty_xL^2_t}\les \|f\|_{L^2_{\xi,\tau}}\,\,\,\text{ (Kato smoothing inequality)}, 
\ee 
\be\label{mf}
\Big\|\Big(\frac{ f(\xi,\tau) }{\la \xi\ra^{\frac14}\la\tau+\xi^2\ra^{\frac12+}}\Big)^\vee \Big\|_{L^4_xL^\infty_t}\les \|f\|_{L^2_{\xi,\tau}}\,\,\,\text{ (Maximal function inequality)}.
\ee
Interpolating both \eqref{ks} and \eqref{mf} with   Plancherel identity we obtain
\be\label{ks-}
\Big\|\Big(\frac{\la \xi\ra^{\frac12-} f(\xi,\tau) }{\la\tau+\xi^2\ra^{\frac12-}}\Big)^\vee \Big\|_{L^{\infty-}_xL^2_t}\les \|f\|_{L^2_{\xi,\tau}},
\ee
\be\label{mf-}
\Big\|\Big(\frac{ f(\xi,\tau) }{\la \xi\ra^{\frac14-}\la\tau+\xi^2\ra^{\frac12-}}\Big)^\vee \Big\|_{L^{4-}_xL^{\infty-}_t}\les \|f\|_{L^2_{\xi,\tau}}.
\ee 
Sobolev embedding implies that 
\be\label{se2}
\Big\|\Big(\frac{ f(\xi,\tau) }{\la \xi\ra^{\frac12-\frac1p+}\la\tau+\xi^2\ra^{\frac12-\frac1p+}}\Big)^\vee \Big\|_{L^p_xL^p_t}\les \|f\|_{L^2_{\xi,\tau}},\,\,\,2\leq p<\infty.
\ee
We also have the Strichartz estimate 
\be\label{str}
\Big\|\Big(\frac{ f(\xi,\tau) }{ \la\tau+\xi^2\ra^{\frac12+}}\Big)^\vee \Big\|_{L^6_xL^6_t}\les \|f\|_{L^2_{\xi,\tau}}.
\ee
Interpolating \eqref{str} with Plancherel identity, we have
\be\label{str-}
\Big\|\Big(\frac{ f(\xi,\tau) }{ \la\tau+\xi^2\ra^{\frac34-\frac3{2p}+}}\Big)^\vee \Big\|_{L^{p}_xL^{p}_t}\les \|f\|_{L^2_{\xi,\tau}},\,\,\,\,\,\,2<p<6.
\ee
Interpolating \eqref{str} with \eqref{se2}, we have
\be\label{str+}
\Big\|\Big(\frac{ f(\xi,\tau) }{ \la \xi\ra^{\frac{p-6}{2p}+}\la\tau+\xi^2\ra^{\frac12-}}\Big)^\vee \Big\|_{L^{p}_xL^{p}_t}\les \|f\|_{L^2_{\xi,\tau}},\,\,\,\,6<p<\infty.
\ee
Using \eqref{str-}, \eqref{str+}, and  H{\oo}lder's inequality, we  have
\be\label{str+-}
\Big\|\Big(\frac{ f(\xi,\tau) }{ \la\tau+\xi^2\ra^{c+}}\Big)^\vee \Big[ \Big(\frac{ f(\xi,\tau) }{ \la \xi\ra^{\frac12-c+}\la\tau+\xi^2\ra^{\frac12-}}\Big)^\vee\Big]^2 \Big\|_{L^{2+}_{x,t }}\les \|f\|_{L^2_{\xi,\tau}}^3,\,\,\,\,\,0\leq c\leq\frac12.
\ee
Similarly, using \eqref{ks-} and \eqref{mf-}, and H{\oo}lder's inequality, we have
\be\label{ksmf-}
\Big\|\Big(\frac{\la \xi\ra^{\frac12-} f(\xi,\tau) }{\la\tau+\xi^2\ra^{\frac12-}}\Big)^\vee \Big[ \Big(\frac{ f(\xi,\tau) }{\la \xi\ra^{\frac14-}\la\tau+\xi^2\ra^{\frac12-}}\Big)^\vee\Big]^2 \Big\|_{L^{2-}_{x,t} }\les \|f\|_{L^2_{\xi,\tau}}^3.
\ee

We start with smoothing estimates for the quintic term in \eqref{eq:duhamel}:
\begin{proposition}\label{prop:smooth} For fixed $s>0$ and $a<\min(4s,\frac{1}{2})$, there exists $\epsilon>0$ such that for $\frac12-\epsilon<b <\frac12$, we have
$$\big\||u|^4u\big\|_{X^{s+a,-b }}\les \|u\|_{X^{s,b }}^5.$$ 
\end{proposition}
 \begin{proof}
 
 By writing the Fourier transform of $|u|^4u =u\bar{u}u\bar{u}u$ as a convolution, we obtain
\[ \widehat{|u|^4u}(\xi_0,\tau_0) = \int\limits_{{\xi_0 - \xi_1 + \xi_2-\xi_3+\xi_4-\xi_5=0  }\atop {\tau_0 - \tau_1+\tau_2-\tau_3+\tau_4-\tau_5 =0}}   \widehat{u}(\xi_1,\tau_1)\overline{\widehat{u}(\xi_2,\tau_2)} \widehat{u}(\xi_3,\tau_3)\overline{\widehat{u}(\xi_4,\tau_4)} \widehat{u}(\xi_5, \tau_5). \]
We define
\[ f(\xi,\tau) = |\widehat{u}(\xi,\tau)|\langle \xi \rangle ^s \langle \tau+ \xi^2 \rangle^{b }. \]
By duality, it suffices to prove
$$
I:=\int\limits_{{\xi_0 - \xi_1 + \xi_2-\xi_3+\xi_4-\xi_5=0  }\atop {\tau_0 - \tau_1+\tau_2-\tau_3+\tau_4-\tau_5 =0}}  \frac{\la\xi_0\ra^{s+a} g(\xi_0,\tau_0) \prod_{j=1}^5 f(\xi_j,\tau_j) }{ \prod_{j=1}^5 \la \xi_j\ra^s \prod_{j=0}^5\la\tau_j+\xi_j^2\ra^b }\les \|f\|_{L^2_{\xi,\tau}}^5\|g\|_{L^2_{\xi,\tau}}.
$$
By symmetry, we can restrict ourselves to the case $|\xi_1|\geq|\xi_2|\geq|\xi_3|\geq|\xi_4|\geq|\xi_5|$, which implies $|\xi_1|\gtrsim|\xi_0|$. We write
\begin{multline*}
I  \les \sup\, \frac{ \la\xi_2\ra^{0+}\la \xi_3\ra^{0+} \la\xi_4\ra^{\frac14-}\la\xi_5\ra^{\frac14-} }{\la\xi_1\ra^{ \frac12-}}  \frac{\la\xi_0\ra^{s+a} }{ \prod_{j=1}^5 \la \xi_j\ra^s  } \\
 \times \int\limits_{{\xi_0 - \xi_1 + \xi_2-\xi_3+\xi_4-\xi_5=0  }\atop {\tau_0 - \tau_1+\tau_2-\tau_3+\tau_4-\tau_5 =0}} \left(\frac{g (\xi_0,\tau_0)   f (\xi_2,\tau_2) f (\xi_3,\tau_3)   }{ \la\xi_2\ra^{0+}\la \xi_3\ra^{0+} \la\tau_0+\xi_0^2\ra^b\la\tau_2+\xi_2^2\ra^b  \la\tau_3+\xi_3^2\ra^b  } \right.\\
\left. \times \frac{ \la\xi_1\ra^{ \frac12-}f (\xi_1,\tau_1) f (\xi_4,\tau_4) f (\xi_5,\tau_5)  }{  \la\xi_4\ra^{\frac14-}\la\xi_5\ra^{\frac14-} \la\tau_1+\xi_1^2\ra^b\la\tau_4+\xi_4^2\ra^b  \la\tau_5+\xi_5^2\ra^b  }\right).
\end{multline*}
Note that for $s>0$, the supremum is finite provided that   $a<\min(1/2,4s) $. Therefore, using \eqref{str+-} with $c=\frac12-$ and \eqref{ksmf-}, as well as the Plancherel identity and the convolution structure,  we get  $ I\les  \|f\|_{L^2_{\xi,\tau}}^5\|g\|_{L^2_{\xi,\tau}}.$ 
 \end{proof} 
We also need the following smoothing bound for the terms arising in Proposition~\ref{prop:duhamelkato}: 
 \begin{proposition}\label{prop:smooth2} For fixed $0< s <\frac52$, and  $\frac12-s<a<\min(4s,\frac{1}{2}, \frac52-s)$,   there exists $\epsilon>0$ such that for $\frac12-\epsilon<b <\frac12$, we have
\begin{align*}  \big\||u|^4u\big\|_{X^{\frac12,\frac{2s+2a-1-4b}{4}}} \les \|u\|_{X^{s,b}}^5. 
\end{align*} 
\end{proposition}
\begin{proof}  As in Proposition~\ref{prop:smooth}, by duality it suffices to prove that
$$
I:=\int\limits_{{\xi_0 - \xi_1 + \xi_2-\xi_3+\xi_4-\xi_5=0  }\atop {\tau_0 - \tau_1+\tau_2-\tau_3+\tau_4-\tau_5 =0}}  \frac{\la\xi_0\ra^{\frac12}  \la \tau_0+\xi_0^2\ra^{\frac{2s+2a-1-4b  }{4}}g(\xi_0,\tau_0) \prod_{j=1}^5 f(\xi_j,\tau_j) }{ \prod_{j=1}^5 \la \xi_j\ra^s \la\tau_j+\xi_j^2\ra^b }\les \|f\|_{L^2_{\xi,\tau}}^5\|g\|_{L^2_{\xi,\tau}}.
$$ 
We proceed by considering three cases. \\
Case 1. $\frac32\leq s+a<\frac52$. Since $ a< \frac12$, we have $s>1$.
Note that 
$$
 \la \tau_0+\xi_0^2\ra \les \la \xi_{max}\ra^2 \max_{j=1,..,5}\la \tau_j+\xi_j^2\ra.
$$
Without loss of generality, let $\max_{j=1,..,5}\la \tau_j+\xi_j^2\ra= \la \tau_5+\xi_5^2\ra$. We have
 $$ \frac{\la \tau_0+\xi_0^2\ra^{\frac{2s+2a-1-4b }{4}} }{ \prod_{j=1}^5   \la\tau_j+\xi_j^2\ra^b }\les \frac{ \la \xi_{max}\ra^{s+a-\frac12-2b}}{ \prod_{j=1}^4   \la\tau_j+\xi_j^2\ra^{\frac12+} }.
 $$
Using this, the Cauchy-Schwarz inequality, and integrating in the $\tau$ variables, we bound $I$ by the square root of 
$$
\|f\|_{L^2_{\xi,\tau}}^{10}\|g\|^2_{L^2_{\xi,\tau}} \sup_{\xi_0}\int\limits_{ \xi_0 - \xi_1 + \xi_2-\xi_3+\xi_4-\xi_5=0   }  \frac{\la\xi_0\ra \la \xi_{max}\ra^{2s+2a-1-4b}   }{ \prod_{j=1}^5 \la \xi_j\ra^{2s} }. 
$$ 
The supremum above is bounded by $ \sup \la \xi_0\ra \la \xi_{max}\ra^{ 2a-1-4b}$, which is finite provided that 
$a\leq 2b$. This completes Case 1.

We now consider the remaining case    $\frac12< s+a<\frac32$. 
By symmetry, we can restrict ourselves to the case $|\xi_1|\geq|\xi_2|\geq|\xi_3|\geq|\xi_4|\geq|\xi_5|$, which implies $|\xi_1|\gtrsim|\xi_0|$. We will consider the cases $s+a<1$ and $s+a\geq1$ separately. In both cases we use the simple observation that $a_1\geq a_2\geq \cdots\geq a_n\geq 1$ and $b_1\geq b_2\geq \cdots\geq b_n $ imply 
$\prod_{j=1}^n a_j^{-b_j} \leq 1$, provided that $\sum_{j=1}^n b_j \geq 0.$\\
Case 2. $\frac12< s+a<1$.  
Using \eqref{ksmf-} for $f(\xi_1,\tau_1)$,   $f(\xi_4,\tau_4)$, and $f(\xi_5,\tau_5)$, and using \eqref{str+-}  with $c:=\frac{3-2s-2a}{4}- \in (\frac14,\frac12)$   for $g(\xi_0,\tau_0)$, $f(\xi_2,\tau_2)$, and $f(\xi_3,\tau_3)$ as in the proof of Proposition~\ref{prop:smooth}, it suffices to observe that  
$$
\sup\, \frac{ \la\xi_2\ra^{\frac12-c+} \la\xi_3\ra^{\frac12-c+} \la\xi_4\ra^{\frac14- }\la\xi_5\ra^{\frac14 -} }{\la\xi_1\ra^{ \frac12 -}}  \frac{\la\xi_0\ra^{\frac12} }{ \prod_{j=1}^5 \la \xi_j\ra^s  } <\infty,
$$
which holds true  for $a<\min(4s,\frac12) $.\\
Case 3. $1\leq s+a<\frac32$.   Using \eqref{ksmf-} for $f(\xi_1,\tau_1)$,  $f(\xi_2,\tau_2)$, and $f(\xi_3,\tau_3)$, and using  \eqref{str+-} with $c:=\frac{3-2s-2a}{4}- \in (0, \frac14)$ for $g(\xi_0,\tau_0)$, $f(\xi_4,\tau_4)$, and $f(\xi_5,\tau_5)$, we have
$$
\sup\, \frac{  \la\xi_2\ra^{\frac14- }\la\xi_3\ra^{\frac14- }\la\xi_4\ra^{\frac12-c+} \la\xi_5\ra^{\frac12-c+}   }{\la\xi_1\ra^{ \frac12- }}  \frac{\la\xi_0\ra^{\frac12} }{ \prod_{j=1}^5 \la \xi_j\ra^s  } <\infty,
$$
since $a<\min(4s,\frac12) $. 
\end{proof}
 
 We finish this section with analogous smoothing estimates for the derivative nonlinearity  in \eqref{eq:duhamel}. The following proposition is a variant of an estimate from \cite{ht1}:
\begin{proposition}\label{prop:smooth3} For fixed $s>\frac12$ and $a<\min(2s-1,\frac{1}{4})$, there exists $\epsilon>0$ such that for $\frac12-\epsilon<b <\frac12$, we have
$$\big\|u^2\overline{u_x}\big\|_{X^{s+a,-b }}\les \|u\|_{X^{s,b }}^3.$$ 
\end{proposition}
\begin{proof}
Passing to the Fourier side and by duality as in the proof of Proposition~\ref{prop:smooth}, it suffices to 
prove
$$
 \int\limits_{{\xi_0 - \xi_1 + \xi_2-\xi_3=0  }\atop {\tau_0 - \tau_1+\tau_2-\tau_3 =0}}  \frac{\la\xi_0\ra^{s+a} \la \xi_2\ra g(\xi_0,\tau_0) \prod_{j=1}^3 f(\xi_j,\tau_j) }{ \prod_{j=1}^3 \la \xi_j\ra^s \prod_{j=0}^3\la\tau_j+\xi_j^2\ra^b }\les \|f\|_{L^2_{\xi,\tau}}^3\|g\|_{L^2_{\xi,\tau}}.
$$
Note that 
$$
 \prod_{j=0}^3  \la\tau_j+\xi_j^2\ra^{b} \gtrsim 
\la (\xi_0-\xi_1)(\xi_0-\xi_3)\ra^{\frac12-} \frac{ \prod_{j=0}^3  \la\tau_j+\xi_j^2\ra^{\frac12+}}{\max\limits_{0\leq j\leq 3 }  \la\tau_j+\xi_j^2\ra^{\frac12+}}.
$$
Using  \eqref{ks} and \eqref{mf} and observing that
$$
\min\Big(\frac{\la \xi_i\ra^{\frac14}\la \xi_j\ra^{\frac14}  }{ \la \xi_k\ra^{\frac12}}, \frac{\la \xi_i\ra^{\frac14}\la \xi_k\ra^{\frac14}  }{ \la \xi_j\ra^{\frac12}}, \frac{\la \xi_j\ra^{\frac14}\la \xi_k\ra^{\frac14}  }{ \la \xi_i\ra^{\frac12}}\Big)\les \frac{\la \xi_i\ra^{\frac14}\la \xi_j\ra^{\frac14} \la \xi_k\ra^{\frac14}}{\la \xi_i\ra^{\frac34}+\la \xi_j\ra^{\frac34}+ \la \xi_k\ra^{\frac34}},
$$ 
 it suffices to prove that 
$$
\sup\limits_{\xi_0 - \xi_1 + \xi_2-\xi_3=0 }\Big(\frac{\la\xi_0\ra^{s+a} \la \xi_1\ra^{-s}\la \xi_2\ra^{1-s}  \la \xi_3\ra^{-s}   }{ \la (\xi_0-\xi_1)(\xi_0-\xi_3)\ra^{\frac12-}   } \max\limits_{0\leq i,j,k\leq 3, \text{ distinct} }  \frac{\la \xi_i\ra^{\frac14}\la \xi_j\ra^{\frac14} \la \xi_k\ra^{\frac14}}{\la \xi_i\ra^{\frac34}+\la \xi_j\ra^{\frac34}+ \la \xi_k\ra^{\frac34}}\Big)\les 1.
$$
Renaming the variables $\xi_0,...,\xi_3$ according to their size as $|\xi_{min}|\leq|\xi_{mid}|\leq|\xi_{max1}|\approx |\xi_{max}| $, we have 
$$
\max\limits_{0\leq i,j,k\leq 3, \text{ distinct} }  \frac{\la \xi_i\ra^{\frac14}\la \xi_j\ra^{\frac14} \la \xi_k\ra^{\frac14}}{\la \xi_i\ra^{\frac34}+\la \xi_j\ra^{\frac34}+ \la \xi_k\ra^{\frac34}} \approx \frac{\la \xi_{mid}\ra^{\frac14}}{\la \xi_{max}\ra^\frac14}.
$$
Therefore, we need to bound 
$$
\sup\limits_{\xi_0 - \xi_1 + \xi_2-\xi_3=0 }\frac{\la\xi_0\ra^{s+a} \la \xi_1\ra^{-s}\la \xi_2\ra^{1-s}  \la \xi_3\ra^{-s}   }{ \la (\xi_0-\xi_1)(\xi_0-\xi_3)\ra^{\frac12-}   } \frac{\la \xi_{mid}\ra^{\frac14}}{\la \xi_{max}\ra^\frac14}.   
$$
The case  $|\xi_0-\xi_1|\les 1$ or   $|\xi_0-\xi_3|\les 1$ is immediate. Thus it suffices to prove  that
\be\label{temp43}
\frac{\la\xi_0\ra^{s+a} \la \xi_1\ra^{-s}\la \xi_2\ra^{1-s}  \la \xi_3\ra^{-s}   }{ \la  \xi_0-\xi_1 \ra^{\frac12-}  \la  \xi_0-\xi_3 \ra^{\frac12-}}  \frac{\la \xi_{mid}\ra^{\frac14}}{\la \xi_{max}\ra^\frac14}\les 1
\ee
when $\xi_0 - \xi_1 + \xi_2-\xi_3=0 $ and $ |\xi_1|\leq |\xi_3|$, by symmetry. 

In the case  $|\xi_1|\leq |\xi_3|\les |\xi_2|\approx |\xi_0|$, we have  
$$
\eqref{temp43}\les \frac{\la\xi_0\ra^{\frac34 +a}      }{  \la \xi_1\ra^{ s}\la  \xi_0-\xi_1 \ra^{\frac12-}  \la  \xi_0-\xi_3 \ra^{\frac12-}  \la \xi_3\ra^{ s-\frac14}}.  $$
For $s\geq \frac34$, we bound this by   $  \la\xi_0\ra^{a-\frac14+} \les 1$ provided that $a<\frac14$. For $\frac12<s<\frac34$, we have the bound  
 $$   \frac{\la\xi_0\ra^{\frac34 +a}      }{  \la \xi_1\ra^{ s-\frac12}\la  \xi_0 \ra^{\frac12-}  \la  \xi_0-\xi_3 \ra^{\frac34-s-}  \la \xi_0\ra^{ s-\frac14}}\les    \frac{\la\xi_0\ra^{\frac12 +a-s+}      }{  \la \xi_2\ra^{\min( s-\frac12,\frac34-s)-}  } \les \la\xi_0\ra^{\frac12 +a-s -\min( s-\frac12,\frac34-s)+}    \les 1  $$ 
provided that $a<\min(\frac14,2s-1)$. In the first inequality above we used $\xi_0-\xi_3=\xi_1-\xi_2$. 

In the case  $|\xi_1|\leq |\xi_3| \approx |\xi_0|$ and $|\xi_3|\gg|\xi_2|$, we have the bound
$$\eqref{temp43}\les  \frac{\la\xi_0\ra^{ a-\frac14} \la \xi_1\ra^{-s}\la \xi_2\ra^{1-s}    (\la \xi_1\ra^{\frac14} +\la \xi_2\ra^{\frac14})}{ \la  \xi_0  \ra^{\frac12-}  \la  \xi_1-\xi_2 \ra^{\frac12-}} \les    
 \la\xi_0\ra^{ a-\frac12+}  \la \xi_2\ra^{\frac12-s+}\les 1
$$
provided that $a<\frac12$.
  
In the case  $|\xi_1|\leq |\xi_3| \approx |\xi_2|$ and $|\xi_3|\gg|\xi_0|$, we have the bound
$$\eqref{temp43}\les  
\frac{\la\xi_0\ra^{s+a}  \la \xi_3\ra^{\frac14-2s+}   (\la \xi_0\ra^{\frac14} +\la \xi_1\ra^{\frac14}) }{ \la \xi_1\ra^{ s} \la  \xi_0-\xi_1\ra^{\frac12-}   } \les     \la\xi_0\ra^{s+a-\frac14+}   \la \xi_3\ra^{\frac14-2s+} \les 1
$$
provided that $a<s$.

In the case  $|\xi_0|,|\xi_2|\ll |\xi_1| \approx |\xi_3|$, we have the bound
$$\eqref{temp43}\les  
 \la\xi_0\ra^{s+a}  \la \xi_2\ra^{1-s}  \la \xi_3\ra^{-2 s-\frac54+} ( \la  \xi_0\ra^{\frac14}+\la  \xi_2\ra^{\frac14})\les  \la\xi_0\ra^{ a-1}  \la \xi_2\ra^{1-2s+}     \les 1 
$$ 
provided that $a\leq 1$. 
\end{proof}
 \begin{proposition}\label{prop:smooth4} For fixed $\frac12< s <\frac52$, and  $ a<\min(2s-1,\frac{1}{4}, \frac52-s)$,   there exists $\epsilon>0$ such that for $\frac12-\epsilon<b <\frac12$, we have
\begin{align*}  \big\| u^2\overline{u_x} \big\|_{X^{\frac12,\frac{2s+2a-1-4b}{4}}} \les \|u\|_{X^{s,b}}^3. 
\end{align*} 
\end{proposition}
\begin{proof} As in Proposition~\ref{prop:smooth}, by duality and letting $b=\frac12-$, it suffices to prove that
$$
 \int\limits_{{\xi_0 - \xi_1 + \xi_2-\xi_3 =0  }\atop {\tau_0 - \tau_1+\tau_2-\tau_3 =0}}  \frac{\la\xi_0\ra^{\frac12} \la \xi_2\ra g(\xi_0,\tau_0) \prod_{j=1}^3 f(\xi_j,\tau_j) }{ \la \tau_0+\xi_0^2\ra^{\frac{3-2s-2a }{4}-}\prod_{j=1}^3 \la \xi_j\ra^s \la\tau_j+\xi_j^2\ra^{\frac12-} }\les \|f\|_{L^2_{\xi,\tau}}^3\|g\|_{L^2_{\xi,\tau}}.
$$ 
First consider the case $s+a<\frac32$. As in the proof of Proposition~\ref{prop:smooth3}, we have 
$$
\la \tau_0+\xi_0^2\ra^{\frac{3-2s-2a }{4}-}\prod_{j=1}^3  \la\tau_j+\xi_j^2\ra^{\frac12-} \gtrsim 
\la (\xi_0-\xi_1)(\xi_0-\xi_3)\ra^{\frac{3-2s-2a }{4}-} \frac{ \prod_{j=0}^3  \la\tau_j+\xi_j^2\ra^{\frac12+}}{\max\limits_{0\leq j\leq 3 }  \la\tau_j+\xi_j^2\ra^{\frac12+}}.
$$
Using  \eqref{ks},  \eqref{mf}, and defining $\xi_{mid}$ and $\xi_{max}$ as in the proof of Proposition~\ref{prop:smooth3}, it suffices to prove that 
$$
\sup\limits_{\xi_0 - \xi_1 + \xi_2-\xi_3 =0  }\frac{\la\xi_0\ra^{\frac12} \la \xi_1\ra^{-s}\la \xi_2\ra^{1-s}  \la \xi_3\ra^{-s}   }{ \la (\xi_0-\xi_1)(\xi_0-\xi_3)\ra^{\frac{3-2s-2a }{4}-} }  \frac{\la \xi_{mid}\ra^{\frac14}}{\la \xi_{max}\ra^\frac14}\les 1.   
$$
The case  $|\xi_0-\xi_1|\les 1$ or   $|\xi_0-\xi_3|\les 1$ is immediate. Thus it suffices to prove that
\be\label{temp444}
\frac{\la\xi_0\ra^{\frac12} \la \xi_1\ra^{-s}\la \xi_2\ra^{1-s}  \la \xi_3\ra^{-s}   }{ \la  \xi_0-\xi_1 \ra^{\frac{3-2s-2a }{4}-}  \la  \xi_0-\xi_3 \ra^{\frac{3-2s-2a }{4}-}}  \frac{\la \xi_{mid}\ra^{\frac14}}{\la \xi_{max}\ra^\frac14}\les 1
\ee
when $\xi_0 - \xi_1 + \xi_2-\xi_3 =0$ and  $|\xi_1|\leq |\xi_3|$, by symmetry. 

In the case  $|\xi_1|\leq |\xi_3|\les |\xi_2|\approx |\xi_0|$, we have  
$$
\eqref{temp444}\les \frac{\la\xi_0\ra^{\frac54 -s}  }{ \la  \xi_0  \ra^{\frac{3-2s-2a }{4}-} \la  \xi_0  \ra^{\min(s-\frac14,\frac{3-2s-2a }{4})-}  } \les 1 
$$
provided that $a<\min(\frac14,2s-1)$.

In the case  $|\xi_1|\leq |\xi_3| \approx |\xi_0|$ and $|\xi_3|\gg|\xi_2|$, we have  
\begin{multline*}
\eqref{temp444}\les \frac{\la\xi_0\ra^{\frac14-s} \la \xi_1\ra^{-s}\la \xi_2\ra^{1-s} \la |\xi_1|+|\xi_2|\ra^{\frac14}     }{ \la  \xi_0  \ra^{\frac{3-2s-2a }{4}-}  \la  \xi_1-\xi_2 \ra^{\frac{3-2s-2a }{4}-} } \\ \les \frac{\la\xi_0\ra^{-\frac14-\frac{s}2+\frac{a}2+} \la \xi_2\ra^{1-s}     }{  \la \xi_1\ra^{ s}  \la  \xi_1-\xi_2 \ra^{\frac{3-2s-2a }{4}-} }
\les \la\xi_0\ra^{-\frac14-\frac{s}2+\frac{a}2+} \la \xi_2\ra^{\frac14-\frac{s}2+\frac{a}2},     
\end{multline*}
which is bounded for $a<s$. 

In the case  $|\xi_1|\leq |\xi_3| \approx |\xi_2|$ and $|\xi_3|\gg|\xi_0|$, we have 
$$
\eqref{temp444}\les \frac{\la\xi_0\ra^{\frac12} \la \xi_1\ra^{-s}\la \xi_2\ra^{ -\frac{3s}{2}+\frac{ a }{2}+}    \la |\xi_{1}|+|\xi_0|\ra^{\frac14}    }{ \la  \xi_0-\xi_1 \ra^{\frac{3-2s-2a }{4}-}  } \les   \la\xi_0\ra^{-\frac14+\frac{s+a}2+}  \la \xi_2\ra^{ -\frac{3s}{2}+\frac{ a }{2}+\frac14+}    \les 1  
$$
provided that $a<s$.

In the case  $|\xi_0|,|\xi_2|\ll |\xi_1| \approx |\xi_3|$, we have 
$$
\eqref{temp444}\les \la\xi_0\ra^{\frac12}  \la \xi_2\ra^{1-s}  \la \xi_3\ra^{a- s-\frac74+}  \la |\xi_0|+|\xi_2|\ra^{\frac14} \les 1 
$$ 
provided that $a<s$.

When $\frac32\leq s+a < \frac52$, we always have $s>1$.  Renaming the variables $\xi_1,\xi_2,\xi_3$ according to their size as $|\xi_{min}|\leq|\xi_{mid}|\leq|\xi_{max }|  $, we have  
$$
 \la \tau_0+\xi_0^2\ra \les \max(\la \xi_{max}\ra^2,  \la \tau_1+\xi_1^2\ra,\la \tau_2+\xi_2^2\ra , \la \tau_3+\xi_3^2\ra).
$$
When the maximum is one of the $\la\tau_j+\xi_j^2\ra$, say $\la\tau_3+\xi_3^2\ra$, we have
\begin{multline*}
\la \tau_0+\xi_0^2\ra^{\frac{3-2s-2a }{4}-}\prod_{j=1}^3   \la\tau_j+\xi_j^2\ra^{\frac12-} 
\gtrsim \la\tau_1+\xi_1^2\ra^{\frac12-}\la\tau_2+\xi_2^2\ra^{\frac12-}\la\tau_3+\xi_3^2\ra^{\frac12+\frac{3-2s-2a }{4}-}\\
\gtrsim  \la\tau_1+\xi_1^2\ra^{\frac12+}\la\tau_2+\xi_2^2\ra^{\frac12+}\la\xi_{max}\ra^{\frac52-s-a-}.
\end{multline*}
Therefore, by the Cauchy-Schwarz inequality (as in the proof of Proposition~\ref{prop:smooth2}), it suffices to prove that
$$
\sup_{\xi_0}\int\limits_{ \xi_0 - \xi_1 + \xi_2-\xi_3 =0   }  \frac{\la\xi_0\ra \la \xi_2\ra^{2}   }{ \la\xi_{max}\ra^{5-2s-2a-}\prod_{j=1}^3 \la \xi_j\ra^{2s} }\les 1. 
$$
Using Lemma~\ref{lem:sums} and $|\xi_{max}|\gtrsim|\xi_0|$, we bound the integral by 
$$\int  \frac{\la\xi_0\ra^{2s+2a-4+}   d\xi_1d\xi_3  }{ \la \xi_1\ra^{2s}\la \xi_0-\xi_1-\xi_3\ra^{2s-2}  \la \xi_3\ra^{2s} }\les \la\xi_0\ra^{2s+2a-4+}  \la \xi_0\ra^{2-2s}  =  \la\xi_0\ra^{ 2a-2+} \les 1 $$
provided that $a<1$. 

When the maximum is $\la \xi_{max}\ra^2$, we have
$$
\la \tau_0+\xi_0^2\ra^{\frac{3-2s-2a }{4}-}\prod_{j=1}^3   \la\tau_j+\xi_j^2\ra^{\frac12-} 
\gtrsim \la  \xi_{max}\ra^{\frac{3-2s-2a }{2}-}\prod_{j=1}^3   \la\tau_j+\xi_j^2\ra^{\frac12+}.
$$
Therefore, using \eqref{ks} and  \eqref{mf} as in the proof of Proposition~\ref{prop:smooth3}, it suffices to prove that 
$$ \la\xi_0\ra^{\frac12} \la \xi_1\ra^{-s}\la \xi_2\ra^{1-s}  \la \xi_3\ra^{-s}   \la \xi_{max}\ra^{s+a-\frac{3}{2}+} \frac{\la\xi_{mid}\ra^{\frac14}}{\la\xi_{max}\ra^{\frac14}}\les 1. 
$$
We bound the multiplier by
$$\frac{ \la \xi_{max}\ra^{s+a-\frac{1}{4}+}\la\xi_{mid}\ra^{\frac14}}{\la \xi_1\ra^{ s}\la \xi_2\ra^{ s}  \la \xi_3\ra^{ s} }\les \frac{ \la \xi_{max}\ra^{s+a-\frac{1}{4}+}\la\xi_{mid}\ra^{\frac14}}{\la \xi_{max}\ra^{ s}\la \xi_{mid}\ra^{ s}  }\les 1 
$$
provided that $a<\frac14$.
\end{proof}

\section{Local theory: The proof of  Theorem~\ref{thm:local} and Theorem~\ref{thm:smooth}} \label{sec:prt3}
We start with the proof of Theorem~\ref{thm:local} for $\alpha=-1$, see \eqref{first}. 
Recalling \eqref{eq:duhamel}, we first prove that
\begin{equation}\label{eq:gamma}
\Gamma u(t):=\eta(t)W_0^t\big(g,h\big)+i \eta(t) \int_0^tW_\R(t- t^\prime)  F(u) \,d t^\prime - \eta(t) W_0^t\big(0, q  \big),
\end{equation}
has a fixed point in $X^{s,b}$. Here  $s\in(\frac12,\frac52)$, $s\neq  \frac32$,    $b<\frac12$ is sufficiently close to $\frac12$, and
$$
F(u)=\eta(t/T)\big( i u^2\overline{u}_x  +\frac12 |u|^4u\big)     \text{ and }  q(t)=\eta(t ) D_0\Big(\int_0^tW_\R(t- t^\prime)  F(u)\, d t^\prime \Big),
$$
and
$$
W_0^t\big(g,h\big)=W_\R (t) g_e +W_0^t\big(0, h-p  \big) ,\,\,\,p(t)=\eta(t) D_0\big(W_\R (t) g_e\big).
$$
By \eqref{eq:xs1}, we have
$$
\|\eta W_\R (t) g_e\|_{X^{s ,b}} \les \|g_e\|_{H^s}\les \|g\|_{H^s(\R^+)}.
$$
Combining \eqref{eq:xs2}, \eqref{eq:xs3}, Proposition~\ref{prop:smooth}, and Proposition~\ref{prop:smooth3}, we obtain
\begin{multline}\label{Fs+abound}
 \big\| \eta(t) \int_0^tW_\R(t- t^\prime)  F(u) \,d t^\prime \big\|_{X^{s+a,\frac12+}} \les \|F(u)\|_{X^{s+a,-\frac12+}} 
  \\ \les T^{\frac12-b-} (\|u^2\overline{u}_x\|_{X^{s+a,-b}} + \||u|^4u\|_{X^{s+a,-b}}) \les T^{\frac12-b-} (\|u\|_{X^{s,b}}^3+ \|u\|_{X^{s,b}}^5).
\end{multline}
Using Proposition~\ref{prop:wbh}, Lemma~\ref{lem:Hs0}  and  noting that the compatibility condition holds we arrive at
\begin{multline*} 
 \| \eta(t) W_0^t\big( 0, h-p  \big)(t)\|_{X^{s,b}}
\les \|(h-p )\chi_{(0,\infty)} \|_{H^{\frac{2s+1}{4}}_t(\R)} \\ \les \|h-p\|_{H^{\frac{2s+1}{4}}_t(\R^+)} 
\les \|h \|_{H^{\frac{2s+1}{4}}_t(\R^+)} +\|p\|_{H^{\frac{2s+1}{4}}_t(\R )}\les \|h \|_{H^{\frac{2s+1}{4}}_t(\R^+)}+\|g\|_{H^s(\R^+)}.
\end{multline*}
In the last inequality, we used Lemma~\ref{lem:kato}.
Finally,
$$
 \| \eta(t) W_0^t\big( 0,  q \big)(t)\|_{X^{s+a,b}}
\les \|q\chi_{(0,\infty)} \|_{H^{\frac{2(s+a)+1}{4}}_t(\R)}   \les  \|q\|_{H^{\frac{2(s+a)+1}{4}}_t(\R )}.
$$
By  Proposition~\ref{prop:duhamelkato}, \eqref{eq:xs3}, and  Propositions  \ref{prop:smooth}, \ref{prop:smooth2}, \ref{prop:smooth3}, \ref{prop:smooth4}, we have
\be\label{qs+abound}
\|q\|_{H^{\frac{2(s+a)+1}{4}}_t(\R )} \les  
\|F\|_{X^{\frac12,\frac{2(s+a)-3+}{4}}}+\|F\|_{X^{s+a,-\frac12+}}  
\les T^{ \frac12-b-}\big(\| u\|_{X^{s,b}}^3+ \| u\|_{X^{s,b}}^5\big).
\ee
Combining these estimates, we obtain
$$
\|\Gamma u\|_{X^{s,b}}\les \|g\|_{H^s(\R^+)}+ \|h \|_{H^{\frac{2s+1}{4}}_t(\R^+)} +  T^{ \frac12-b-}\big(\| u\|_{X^{s,b}}^3+ \| u\|_{X^{s,b}}^5\big).
$$
 
This yields the existence of a fixed point $u$ of $\Gamma$ in $X^{s,b}$.
Now we prove that $u\in C^0_t H^s_x([0,T)\times \R)$. Note that the first term in the definition \eqref{eq:gamma} of $\Gamma$ is continuous in $H^s$.  The continuity of the third term
follows from Lemma~\ref{lem:wbcont}  and \eqref{qs+abound}. For the second term it follows from the embedding  $X^{s,\frac12+}\subset C^0_tH^s_x$
 and \eqref{eq:xs2} together with Proposition~\ref{prop:smooth}. The fact that  $u\in C^0_x H^{\frac{2s+1}{4}}_t(\R\times [0,T])$ follows similarly from Lemma~\ref{lem:kato}, Proposition~\ref{prop:duhamelkato}, and Lemma~\ref{lem:wbcont}.

The continuous dependence on the initial and boundary data follows from the fixed point argument and the a priori estimates as in the previous paragraph. 
In order to do this observe that by Lemma~\ref{lem:ext} below,  given  $g_n\to g$ in $H^s(\R^+)$ and an $H^s$ extension $g_e$ of $g$, there are extensions $g_{n,e}$ of the $g_n$ so that 
$g_{n,e}\to g_e$ in $H^s(\R)$. 

It remains to prove the uniqueness part of the theorem. Before that, we prove Theorem~\ref{thm:smooth}.  Recall that 
$$
u-W_0^t(g,h)=- \eta(t) W_0^t\big( 0,  q \big)(t)+i\eta(t) \int_0^tW_\R(t- t^\prime)  F(u) dt^\prime.
$$
By \eqref{Fs+abound} and  the embedding  $X^{s+a,\frac12+}\subset C^0_tH^{s+a}_x$, the second summand is in $C^0_tH^{s+a}_x$. The first summand also belongs to $C^0_tH^{s+a}_x$ by Lemma~\ref{lem:wbcont} and  \eqref{qs+abound}.

\subsection{Uniqueness of solutions}\label{sec:unique} We now discuss the uniqueness of solutions of \eqref{first}.
The solution we constructed above is the unique fixed point of \eqref{eq:gamma}. However, it is not a priori clear if there are no other distributional solutions, or if different extensions of the initial data produce  the same solution on $\R^+$.   
We resolve this issue for $s\geq 2$ first.  Take two $H^s$ local solutions $u_1$, $u_2$. 
We have
\be \label{nls4}\left\{\begin{array}{l}
i(u_1-u_2)_t+(u_1-u_2)_{xx}+i u_1^2\overline{u_1}_x +\frac12|u_1|^4u_1-i u_2^2\overline{u_2}_x - \frac12|u_2|^4u_2 =0,\\
(u_1-u_2)(x,0)=0, \,\,\,\,(u_1-u_2)(0,t)=0,\,\,\,\,x\in\R^+, t\in \R^+.  
\end{array}\right.
\ee
Multiplying the equation by $\overline{u_1-u_2}$ and integrating, we obtain (in the local existence interval)
$$
\partial_t\|u_1-u_2\|_2^2\les \|u_1-u_2\|_2^2 \big(1+\|u_1\|_{H^{2}}^3+\|u_2\|_{H^{2}}^3\big)\les  \|u_1-u_2\|_2^2.
$$
This  implies uniqueness for $s\geq 2$. 
 
Implementing the smoothing   bound   in Theorem~\ref{thm:smooth} we now prove the uniqueness for $\frac74<s<2$; the argument can be iterated to obtain uniqueness for all $s\in(\frac12,2), s\neq \frac32$.   
Let $u, v$ be two $H^s(\R^+)$ solutions as in Definition~\ref{def:lwp} starting from data $g,h$. 
Take sequences $\{g_n\}\subset H^2(\R^+)$ converging to $g$ in $H^s(\R^+)$, and $\{h_n\}\subset H^{\frac54}(\R^+)$ converging to $h$ in $H^{\frac{2s+1}4}(\R^+)$. Let $u_n$ be the unique $H^2(\R^+)$
solution on $[0,T_n]$. By continuous dependence this implies that $u=v=\lim u_n$ provided that  the existence times, $T_n$, do  not shrink to zero. To see that $T_n=T_n(\|g\|_{H^s},\|h \|_{H^{\frac{2s+1}4}})$, we use the smoothing estimates noting that $u_n$ agrees with the solution we obtained. Write
\begin{multline*}
\|\Gamma(u_n)\|_{X^{2,b}}\les  \|f_n\|_{H^{2}}+\|h_n\|_{H^{\frac54}}+T_n^{0+} (\|u_n\|_{X^{s,b}}^5+1)\\
\les  \|f_n\|_{H^{2}}+\|h_n\|_{H^{\frac54}}+T_n^{0+}  ( \|f_n\|_{H^{s}}+\|h_n\|_{H^{\frac{2s+1}4}} +1)^5\\
\les  \|f_n\|_{H^{2}}+\|h_n\|_{H^{\frac54}}+T_n^{0+}  ( \|g \|_{H^{s}}+\|h \|_{H^{\frac{2s+1}4}} +1)^5.
\end{multline*}
Since $\|f_n\|_{H^{2}}\geq  \|f_n\|_{H^{s}}\gtrsim \|g\|_{H^s}$, and similarly for $h_n$, we can pick $T_n$ depending only on $\|g\|_{H^s}$ and $\|h \|_{H^{\frac{2s+1}4}}$.

We finish this section with a lemma that was used above which also implies that the solutions we constructed  are limits of $H^{\frac52-}$ solutions.
\begin{lemma}\label{lem:ext} Fix $ s\geq 0$ and $k\geq s$. 
Let $g\in H^s(\R^+)$, $f\in H^k(\R^+)$, and let $g_e$ be an $H^s$ extension of $g$ to $\R$. Then there is an $H^k$ extension $f_e$ of $f$ to $\R$ so that 
$$\|g_e-f_e\|_{H^s(\R)}\les \|g-f\|_{H^s(\R^+)}.
$$ 
\end{lemma}
\begin{proof}  Fix $\psi \in H^s(\R)$ supported in $(-\infty,0]$. We claim that for any $\epsilon>0$, there is a function $\phi\in H^k(\R)$ supported in $(-\infty,0)$ such that $\|\phi-\psi\|_{H^s(\R)}<\epsilon$. Indeed, observing that 
 $\tau_{-\delta}\psi(\cdot)=\psi(\cdot+\delta) \to \psi(\cdot) $ 
in $H^s(\R)$ as $\delta \to 0^+$,    the claim follows by taking a smooth approximate identity $k_n$ supported in $(-\frac\delta2,\frac\delta2)$ for sufficiently small $\delta$, and letting
$\phi= (\tau_{-\delta} \psi ) * k_n$ for sufficiently large $n$. 

To obtain the lemma from this claim, let $\widetilde f$ be an $H^k$ extension of $f$ to $\R$ and   $h$   an $H^s$ extension of $g-f$ to $\R$ with $\|h\|_{H^s(\R)}\les \|g-f\|_{H^s(\R^+)}$. Apply the claim to $\psi=g_e-\widetilde f-h$ with $\epsilon=\|g-f\|_{H^s(\R^+)}$. Letting $f_e=\widetilde f +\phi$ yields the claim.  
\end{proof}

\subsection{The proof of  Theorem~\ref{thm:local} for general ${\mathbb \alpha\in \R}$}\label{sec:genalpha}
In this section we obtain  the local wellposedness of derivative NLS  on $\R^+$:  
\be \label{dnls}\left\{\begin{array}{l}
 iq_t+q_{xx}- i(|q|^2q)_x  =0,\,\,\,\,x\in\R^+, t\in \R^+,\\
 q(x,0)=G(x), \,\,\,\,q(0,t)=H(t).  \end{array}\right.
\ee
The same argument applies to any other gauge $\mG_\alpha$.

Let 
$$
u(x,t)=e^{i\int_x^\infty  |q(y,t) |^2dy} q(x,t) = e^{i\int_x^\infty  |u(y,t) |^2dy} q(x,t),\,\,\,\,\,t,x\in\R^+.
$$
We note that if $u$ solves \eqref{first} with $g(x)=e^{i\int_x^\infty  |G(y) |^2dy} G(x)$, and $h(t)=e^{i\int_0^\infty  |u(y,t) |^2dy}H(t)$, then $q$ solves \eqref{dnls}. Note that this constitutes  a different  boundary value problem than \eqref{first}, since the boundary value depends on the value of the function in the interior of the domain. Hence, the essential part of this process is finding  $h$ of the form $e^{i\gamma(t)} H(t)$, so that the solution $u$ of \eqref{first}, with data $g$, $h$, satisfies 
$$
\int_0^\infty  |u(y,t) |^2dy= \gamma(t),\,\,\,\,t\in[0,T].
$$
The following lemmas establish this and  allow us to obtain the local wellposedness of \eqref{dnls}.
\begin{lemma}\label{dnlsgauge}
Given $G\in H^s(\R^+)$ and $H\in H^{\frac{2s+1}4}(\R^+)$, there is a unique real valued function $\gamma\in H^{\frac{2s+1}4}([0,T])$ such that  the solution $u$ of \eqref{first} with data $g(x)=e^{i\int_x^\infty  |G(y) |^2dy} G(x)$, and $h(t)=e^{i\gamma(t)}H(t)$, 
satisfies 
$$
\gamma(t)= \int_0^\infty  |u(y,t) |^2dy,\,\,\,t\in[0,T].
$$
Here $T=T(\|G\|_{H^s},\|H\|_{H^{\frac{2s+1}4}})$. Moreover, $\gamma$ depends on $G$ and $H$ continuously. Furthermore, $\gamma\in H^1([0,T])$ for $s\in(\frac12,\frac32)$ and $\gamma\in H^{\frac32}([0,T])$ for $s\in(\frac32,\frac52)$.
\end{lemma}

\begin{proof} Fix $G$ and $H $ as in the statement.   Given a real valued function $\gamma\in H^{\frac{2s+1}4}([0,T])$, denote the solution $u$ by $u^\gamma$. 
We will prove the theorem by applying a  fixed point argument to the map 
$$f(\gamma)= \|u^\gamma\|_{L^2(\R^+)}^2 $$
for sufficiently small $T$. 
Let 
$$K_{r,T}=\{\phi\in H^r([0,T]): \|\phi\|_{H^r([0,T])} \leq 1, \phi(0)=\|g\|_{L^2(\R^+)}^2 \}.$$
The following claim finishes the proof:

Claim. For $s \in (\frac12,\frac32)$, there exists $T>0$ as in the statement of the theorem so that   $f$ maps $K_{\frac{2s+1}4,T}$ to $K_{1,T}$, and it is a contraction on  $K_{1,T}$.  Similarly,
for $s \in (\frac32,\frac52)$, $f$ maps $K_{\frac{2s+1}4,T}$ to $K_{\frac32,T}$, and it  is a contraction on $K_{\frac32,T}$.
 
Proof of the Claim. Fix $s \in (\frac12,\frac32)$. By the local theory, 
$$\|f(\gamma)\|_{L^2([0,T])}\les T^{\frac12} C_{ \|G \|_{H^{s}} ,\|H\|_{H^{\frac{2s+1}{4}}}, \|\gamma\|_{H^{\frac{2s+1}{4}}}}.$$
Therefore it suffices to consider  $ \|\partial_t f(\gamma)\|_{L^2([0,T])} $.
We calculate
\be\label{ft}
\partial_t f(\gamma)= 2\Im \big(\overline{h(t)} u^\gamma_x(0,t)\big) -\frac12 |h(t)|^4.
\ee 
Recall that
$  u^\gamma$ solves the equation 
$$
u^\gamma(t)=W_0^t(g,h)+\eta(t) \int_0^tW_\R(t- t^\prime)  \eta(t^\prime /T) N(u^\gamma)\,d t^\prime - \eta(t) W_0^t (0,  q^\gamma ),
$$
where
\begin{align*}
 q^\gamma(t)=\eta(t ) D_0\Big(\int_0^tW_\R(t- t^\prime)  \eta(t^\prime /T) N(u^\gamma) \, d t^\prime \Big).
\end{align*}
Therefore 
$$
u^\gamma_x(0,t)= D_0\partial_x\Big( W_0^t(g,h)+\eta(t) \int_0^tW_\R(t- t^\prime)  \eta(t^\prime /T) N(u^\gamma)\,d t^\prime - \eta(t) W_0^t (0,  q^\gamma )\Big).
$$
 We bound the  $H^\epsilon$ norm (for $0<\epsilon<\frac{2s-1}{4}$) as follows
\begin{multline*}
\|u^\gamma_x(0,t)\|_{H^{\epsilon}}\les \| g \|_{H^{\frac12+2\epsilon}} + \|h\|_{H^{\frac12+\epsilon}}+ \|\eta(t  /T) N(u^\gamma)\|_{X^{\frac12+2\epsilon ,-\frac12+}} +\|q^\gamma\|_{H^{\frac12+\epsilon}}\\
\les \| g \|_{H^{s}} + \|h\|_{H^{\frac{2s+1}{4}}}+\|\eta(t  /T) N(u^\gamma)\|_{X^{s ,-\frac12+}}  \les C_{ \| g \|_{H^{s}} , \|h\|_{H^{\frac{2s+1}{4}}}}.
\end{multline*}
In the first inequality we used Lemma~\ref{lem:kato}, Lemma~\ref{lem:wbcont}, and Proposition~\ref{prop:duhamelkato}. The last inequality follows from the local theory. 

Using this bound in \eqref{ft}, we obtain
\begin{multline*}
\|\partial_t f(\gamma)\|_{L^2([0,T])} \les \|h\|_{L^{\infty-}([0,T])} \|u^\gamma_x(0,t)\|_{L^{2+}}+\|h^4\|_{L^2([0,T])}\\
\les T^{0+} \|H\|_{L^\infty} \|u^\gamma_x(0,t)\|_{H^{\epsilon}}+T^{\frac12}\|H\|_{L^\infty}^4 \les
T^{0+} C_{ \|G \|_{H^{s}} ,\|H\|_{H^{\frac{2s+1}{4}}}, \|\gamma\|_{H^{\frac{2s+1}{4}}}}.
\end{multline*}
In the last step we used 
$$
\|g\|_{H^s}\les \|G\|_{H^s},  \text{ and}
$$
$$
\|h\|_{H^{\frac{2s+1}4}([0,T])}\les \|e^{i\gamma}\|_{H^{\frac{2s+1}4}([0,T])} \|H\|_{H^{\frac{2s+1}4}([0,T])}\les \exp\big(\|\gamma \|_{H^{\frac{2s+1}4}([0,T])}\big)  \|H\|_{H^{\frac{2s+1}4}([0,T])}.
$$ 
This and similar bounds for the differences complete the proof for $s \in (\frac12,\frac32)$ by choosing $T$ small. 

For $s \in (\frac32,\frac52)$, it suffices to consider $ \|\partial_t f(\gamma)\|_{H^{\frac12}([0,T])} $. Interpolating with the $L^2$ bound above, it suffices to prove that  
$$ \|\partial_t f(\gamma)\|_{H^{\frac12+\epsilon}([0,T])} \leq
 C_{ \|G \|_{H^{s}} ,\|H\|_{H^{\frac{2s+1}{4}}}, \|\gamma\|_{H^{\frac{2s+1}{4}}}},$$
which follows from similar arguments using Lemma~\ref{lem:kato}, Lemma~\ref{lem:wbcont},  Proposition~\ref{prop:duhamelkato}, and the local theory.
\end{proof}
\begin{rmk}
Note that for $s\in(\frac12,\frac32)$, a slight variation of the proof above utilizing the fractional Leibniz  rule   and Sobolev embedding theorem yields  that $\gamma\in H^{1+\epsilon}$ for some $\epsilon=\epsilon(s)>0$.
\end{rmk}
\begin{lemma}\label{cont_gauge} Fix $s\in(\frac12,\frac52)$, $s\neq \frac32$. Let $u$ be an $H^s$ solution of \eqref{first}. Then for any $\alpha\in \R$
$$
e^{i\alpha \int_x^\infty|u(y,t)|^2dy} u(x,t)\in C^0_tH^s_x([0,T]\times \R )\cap C^0_xH^{\frac{2s+1}4}_t( \R \times [0,T]).
$$
\end{lemma}
\begin{proof} Since $u\in C^0_tH^s_x([0,T]\times \R )\cap C^0_xH^{\frac{2s+1}4}_t( \R \times [0,T])$,
the first inclusion follows from the Lipschitz continuity of the gauge transformation, see Lemma~\ref{lem:contgauge}
in the Appendix, also see \cite{ckstt}.  
For the second inclusion first note that
$$
\int_x^\infty|u(y,t)|^2dy\in H^{\frac{2s+1}4}_t 
$$
for $x=0$ by  Lemma~\ref{dnlsgauge}; the proof is identical for fixed $x\neq 0$. By Taylor expansion and 
the algebra property of Sobolev spaces, this implies that
$$
e^{i\alpha \int_x^\infty|u(y,t)|^2dy} u(x,t)\in H^{\frac{2s+1}4}_t  
$$
for all $x$. For the continuity in $x$ we consider the differences. By the algebra property of Sobolev spaces and considering the Taylor expansion of exponentials, it suffices to  note that 
\begin{multline*}
\Big\|\int_x^\infty|u(y,t)|^2dy -\int_{x^\prime}^\infty|u(y,t)|^2dy\Big\|_{H^{\frac{2s+1}4}_t}\les 
\int_x^{x^\prime} \big\||u(y,t)|^2\big\|_{H^{\frac{2s+1}4}_t}dy \\
\les
 \int_x^{x^\prime} \big\| u(y,t) \big\|_{H^{\frac{2s+1}4}_t}^2dy \to 0,
\end{multline*}
as $x^\prime\to x.$
\end{proof}

It is easy to see that using Lemma~\ref{dnlsgauge} and Lemma~\ref{cont_gauge} one can construct solutions of \eqref{dnls} in  $C^0_tH^s_x([0,T]\times \R )\cap C^0_xH^{\frac{2s+1}4}_t( \R \times [0,T])$
which depend continuously on the data. Note that the local existence time  depends on $T$ in Lemma~\ref{dnlsgauge}, which depends only on $\|G\|_{H^s}$, $\|H\|_{H^\frac{2s+1}4}$.

It remains to establish the uniqueness of the solutions, which follows from the  previous argument for $s\geq 2$; see the discussions around equation \eqref{nls4} in Section~\ref{sec:unique}. For $s<2$, the smooth approximation argument in Section~\ref{sec:unique} requires  that for smooth $G_n$, $H_n$ converging to $G$, $H$ in $H^s$, $ H^{\frac{2s+1}4}$ respectively,  we find $\gamma_n$ as in Lemma~\ref{dnlsgauge} on an interval $[0,T_n]$, with $T_n=T_n(\|G\|_{H^s},\|H\|_{H^\frac{2s+1}4})$. This follows from Lemma~\ref{dnlsgauge} and the remark following its proof, since for $s\in(\frac32,\frac52)$, say,  we can construct $\gamma_n\in H^{\frac32}$ on $[0,T_n]$, with 
$T_n$ depending only on $\|G_n\|_{H^s}$, $\|H_n\|_{H^{\frac{2s+1}4}}.$ 
 
\section{Global wellposedness in the energy space}\label{sec:global}
In this section we prove Theorem~\ref{thm:global}. We first consider the global wellposedness of \eqref{DNLS_Ga} for   $\alpha=-\frac12$:
\be\label{a-12}\left\{\begin{array}{l}
iu_t+u_{xx}-i|u|^2u_x=0,\,\,\,\,x\in\R^+, t\in \R^+,\\
u(x,0)=g(x), \,\,\,\,u(0,t)=h(t).  
\end{array}\right. 
\ee
Noting the identity
\be\label{enerid}
4\Im(|u|^2u_x\overline{u}_t)=\partial_t\Im(\overline{u}u_x|u|^2)-\partial_x\Im(\overline{u}u_t|u|^2),
\ee
one can easily prove that  the following energy functional is conserved for the problem on $\R$:
$$E_{-\frac12}(u)=\|u_x\|_{L^2(\R)}^2+\frac12\Im\int_{\R} u\overline{u}_x|u|^2.
$$
 
Substituting $u=\mG_{-\frac12} q$, we see that the energy for derivative NLS on $\R$ is
$$
E(q)=\|q_x\|_{L^2(\R)}^2+\frac32\Im\int_{\R}  q\overline{q}_x|q|^2+\frac12\|q\|_{L^6(\R)}^6.
$$
 In what follows,  with a slight abuse of notation, we use $E_{-\frac12}$ and  $E$ also to denote    the functionals where  $\R$ is  replaced by $\R^+$.   
  
To prove the global wellposedness of \eqref{a-12} in the energy space we need to find an a priori bound for the $H^1$ norm of the solution. To this end we use the following identities which can be obtained using \eqref{enerid} and justified by approximation by $H^2$ solutions of \eqref{a-12}:
\begin{align}
\label{eq:m} &\partial_t|u|^2= -2 \Im(u_x\overline{u})_x+\frac12(|u|^4)_x,\\
\label{eq:e}&\partial_t\big(|u_x|^2+\frac12 \Im(\overline{u}_xu |u|^2)\big)=2\Re(u_x\overline{u_t})_x-\frac12\Im(\overline{u}u_t|u|^2)_x,\\
\label{eq:l}&\partial_x(|u_x|^2)=-i\big[(u\overline{u_x})_t-(u\overline{u_t})_x\big].
\end{align}
Similar identities were used  in \cite{bonaetal,ETNLS} for the  NLS equation on the half line. 
Integrating these identities on $[0,t]\times[0,\infty)$, we obtain
\begin{align*}
\|u(t)\|_2^2-\|g\|_2^2&=2\Im\int_0^tu_x(0,s)\overline{h}(s)ds-\frac12\int_0^t|h(s)|^4ds,\\ 
E_{-\frac12}(u(t))-E_{-\frac12}(g)&=-2\Re\int_0^tu_x(0,s)\overline{h^\prime}(s)ds+\frac12\Im\int_0^t \overline{h(s)}h^\prime(s)|h(s)|^2ds,\\ 
I_t:=\int_0^t|u_x(0,s)|^2ds&=i\int_0^\infty u\overline{u}_x dx -i \int_0^\infty g \overline{g^\prime} dx +i \int_0^th(s)\overline{h^\prime}(s)ds.
\end{align*}
Also note that by the Gagliardo-Nirenberg inequality one can obtain
$$E_{-\frac12}(u)\geq  \|u_x\|_{L^2}^2(1-C\|u\|_{L^2}^2),
$$
for some absolute constant $C$. Therefore, in the case $\|u\|_{L^2}^2\leq \frac1{2C}$, we have
\begin{align*}
\|u \|_{L^2}^2&\leq c(1+\sqrt{I_t}), \\
\|u_x \|_{L^2}^2&\leq  c(1+\sqrt{I_t}). \\
I_t&\leq \|u\|_{L^2} \|u_x\|_{L^2} + c,
\end{align*}
where $c= c(C)\leq 1$ depends on $\|g\|_{H^1}+\|h\|_{H^1}$. 
Combining these inequalities, we obtain 
$$
I_t\leq   2c +c\sqrt{I_t}.
$$
We conclude that $I_t\leq 4c$, and hence $\|u \|_2^2\leq 2c, \|u_x \|_2^2\leq 2c$.
This implies that there exists an absolute constant $c>0$ so that \eqref{a-12} is globally wellposed in $H^1(\R^+)$ provided that $\|g\|_{H^1(\R^+)}+\|h\|_{H^1(\R^+)}\leq c$. This yields Theorem~\ref{thm:global} for $\alpha=-\frac12$.  

To obtain the global wellposedness of derivative NLS\footnote{For other values of $\alpha$ the proof is similar.}, plug $u=\mG_{-\frac12}q$ in  the identities \eqref{eq:m}--\eqref{eq:l}, and integrate on $[0,t]\times[0,\infty)$ to  obtain
 $$
\|q(t)\|_2^2-\|G\|_2^2=2\Im\int_0^t(\mG_{-\frac12}q)_x(0,s)\overline{(\mG_{-\frac12}q)}(0,s)ds-\frac12\int_0^t|H(s)|^4ds,
$$
\begin{multline*}
E(q(t))-E(G)=-2\Re\int_0^t (\mG_{-\frac12}q)_x(0,s)\overline{(\mG_{-\frac12}q)_s}(0,s)ds\\
+\frac12\Im\int_0^t \overline{(\mG_{-\frac12}q)}(0,s) (\mG_{-\frac12}q)_s(0,s) |H(s)|^2ds,
\end{multline*}
\begin{multline*}
I_t:=\int_0^t| (\mG_{-\frac12}q)_x(0,s)|^2ds=i\int_0^\infty  (\mG_{-\frac12}q)\overline{ (\mG_{-\frac12}q)}_x dx \\ -i \int_0^\infty \mG_{-\frac12}G \overline{(\mG_{-\frac12}G)^\prime} dx +i \int_0^t(\mG_{-\frac12}q)(0,s)\overline{(\mG_{-\frac12}q)_s}(0,s) ds.
\end{multline*}
In addition,  the definition of the gauge transformation and the $x$-integral of \eqref{eq:m} give 
$$
| (\mG_{-\frac12}q)_s (0,s)|\les |H^\prime(s)|+|H(s)|^5+|H(s)|^2| (\mG_{-\frac12}q)_x(0,s)|.
$$
Furthermore, by the boundedness of the gauge  transformation in $H^1$ and in $L^2$, and the Gagliardo-Nirenberg inequality we have the lower bound for the energy  as above. We thus obtain for small data
\begin{align*}
\|q\|_{L^2}^2&\leq c(1+\sqrt{I_t}),\\ 
\|q_x \|_{L^2}^2&\leq  c(1+\sqrt{I_t}+I_t),\\
I_t&\les c(1+\sqrt{I_t}+ I_t^{\frac34}).
\end{align*}
This concludes the proof of Theorem~\ref{thm:global} for $\alpha=0$.

\section{Derivative NLS on the real line}\label{sec:DNLSR}

In this section we prove Theorem~\ref{thm:smoothR} providing an improved smoothing estimate on the real line by applying  a normal form transform to the equation \eqref{DNLS-1}. 
Following the differentiation by parts method of \cite{bit} (also see \cite{egt} for an application on $\R$), we obtain
 \begin{proposition}\label{prop:normal} The solution $u$ of equation \eqref{DNLS-1} satisfies 
 \be\label{preduhamelR}
i\partial_t \left(e^{-it \Delta } u - e^{-it \Delta }  B (u)  \right)= - e^{-it \Delta } \big(R(u) +\frac12|u|^4u+ NR_1(u) +NR_2(u)  \big),
\ee
where
$$
\widehat{B (u)}(\xi)= \int\limits_{\xi-\xi_1+\xi_2-\xi_3=0 \atop{|\xi-\xi_1|, |\xi-\xi_3|\geq 1}} \frac{\xi_2 u_{\xi_1}\overline{u_{\xi_2}}u_{\xi_3}}{\xi^{2 }-\xi_1^{2 } +\xi_2^{2 }- \xi_3^{2 }},  
$$
$$
\widehat{R(u)}(\xi)=\int\limits_{\xi-\xi_1+\xi_2-\xi_3=0 \atop{|\xi-\xi_1|<1 \text{ or } |\xi-\xi_3|<1 }} \xi_2 u_{\xi_1}\overline{u_{\xi_2}}u_{\xi_3},
$$
$$
\widehat{NR_1(u)}(\xi)=2\int\limits_{\xi-\xi_1+\xi_2-\xi_3=0 \atop{|\xi-\xi_1| , |\xi -\xi_3|\geq 1 }} \frac{\xi_2 u_{\xi_1}\overline{u_{\xi_2}}w_{\xi_3} }{\xi^{2 }-\xi_1^{2 } +\xi_2^{2 }- \xi_3^{2 }}, 
$$
$$
\widehat{NR_2(u)}(\xi)=- \int\limits_{\xi-\xi_1+\xi_2-\xi_3=0 \atop{|\xi-\xi_1| , |\xi -\xi_3|\geq 1 }} \frac{\xi_2 u_{\xi_1}\overline{w_{\xi_2}} u_{\xi_3} }{\xi^{2 }-\xi_1^{2 } +\xi_2^{2 }- \xi_3^{2 }}.
$$
Here $u_\xi(t):=u(\widehat\xi,t)$ and  $w_\xi(t):=w(\widehat\xi,t)$ with 
$$
w=ie^{it\Delta}[\partial_t(e^{-it\Delta}u)]=-  iu^2\overline{u_x}-\frac12 |u|^4u.
$$
\end{proposition}
\begin{proof}
The following calculations can be justified  by smooth approximation. First observe that
$$
i \partial_t(e^{-it\Delta}u) =  e^{-it\Delta} \big(iu_t+u_{xx}\big)= -i  e^{-it\Delta} (u^2\overline{u_x}) -\frac12  e^{-it\Delta}(|u|^4u).
$$
On the Fourier side, we have
\begin{multline*}
-i \mF(e^{-it\Delta}(u^2\overline{u_x}))(\xi)=-\int\limits_{\xi-\xi_1+\xi_2-\xi_3=0} e^{it\xi^2} \xi_2 u_{\xi_1}\overline{u_{\xi_2}}u_{\xi_3}\\
=-e^{it\xi^2} \widehat{R(u)}(\xi)- \int\limits_{\xi-\xi_1+\xi_2-\xi_3=0 \atop{|\xi-\xi_1|, |\xi-\xi_3|\geq 1}}  e^{it\xi^2} \xi_2 u_{\xi_1}\overline{u_{\xi_2}}u_{\xi_3}. 
\end{multline*}
We rewrite the   integral above as
\begin{multline*}
-i\partial_t\big(e^{it\xi^2} \widehat{B(u)}(\xi)\big)+i \int\limits_{\xi-\xi_1+\xi_2-\xi_3=0 \atop{|\xi-\xi_1|, |\xi-\xi_3|\geq 1}}\frac{e^{it(\xi^2-\xi_1^2+\xi_2^2-\xi_3^2)} \xi_2  \partial_t\big(e^{it\xi_1^2}u_{\xi_1}\overline{e^{it\xi_2^2}u_{\xi_2}}e^{it\xi_3^2}u_{\xi_3}\big)}{\xi^{2 }-\xi_1^{2 } +\xi_2^{2 }- \xi_3^{2 }}\\
=-i\partial_t\big(e^{it\xi^2} \widehat{B(u)}(\xi)\big)+e^{it\xi^2}\widehat{NR_1(u)}(\xi) +e^{it\xi^2}\widehat{NR_2(u)}(\xi).
\end{multline*}
The equality follows from the definition of $w$ and the symmetry in $\xi_1$, $\xi_3$ variables.
\end{proof}
The following proposition estimates the terms that appear  in \eqref{preduhamelR}.  
\begin{proposition}\label{prop:smooth5} For fixed $s> \frac12 $, and  $ a<\min(2s-1,\frac12)$,   we have
\begin{align*}  
&\|B(u)\|_{H^{s+a}}\les \|u\|_{H^s}^3,\\
&\|R(u)\|_{X^{s+a,-\frac12+}}\les \|u\|_{H^s}^3,\\
& \big\| |u|^4u \big\|_{X^{ s+ a,-\frac12+ }} \les \|u\|_{X^{s,\frac12+}}^5, \\
& \big\| NR_1(u)+NR_2(u) \big\|_{X^{ s+ a,-\frac12+ }} \les \|u\|_{X^{s,\frac12+}}^2 \|w\|_{X^{s,-\frac38}},
\end{align*} 
where $w= -  iu^2\overline{u_x}-\frac12 |u|^4u$.
\end{proposition}
\begin{proof} First note that the bound for $|u|^4u$ follows from Proposition~\ref{prop:smooth}.
 
By writing $f(\xi)=\la\xi\ra^s |u_\xi|$, the following inequality implies the bound for $B$:  
$$
 \Big\|\int\limits_{\xi-\xi_1+\xi_2-\xi_3=0 } \frac{\la \xi\ra^{s+a}\la\xi_2 \ra^{1-s} f(\xi_1)f(\xi_2)f(\xi_3) }{\la \xi-\xi_1\ra \la \xi-\xi_3\ra \la\xi_1\ra^s  \la\xi_3\ra^s} \Big\|_{L^2_\xi}\les \|f\|_{L^2}^3.
$$
By the Cauchy-Schwarz inequality and symmetry, this boils down to showing 
$$
\sup_\xi \int\limits_{\xi-\xi_1+\xi_2-\xi_3=0\atop{|\xi_3|\leq |\xi_1|} } \frac{\la \xi\ra^{2s+2a}\la\xi_2 \ra^{2 -2s} }{\la \xi-\xi_1\ra^2 \la \xi-\xi_3\ra^2 \la\xi_1\ra^{2s}  \la\xi_3\ra^{2s}}<\infty.
$$ 
This supremum is finite by considering the cases $|\xi_1|\gtrsim |\xi_2|  \gg|\xi|$, $|\xi_2|\approx |\xi|$, 
$|\xi_1|\gtrsim |\xi|\gg|\xi_2|$.

By duality, renaming the functions, and symmetry,  the bound for $R$ follows from  
$$ 
\int\limits_{{\xi_0 - \xi_1 + \xi_2-\xi_3 =0,  |\xi_0 - \xi_1 |<1 }\atop {\tau_0 - \tau_1+\tau_2-\tau_3  =0}  } 
 \frac{\la\xi_0\ra^{s+a} \la \xi_2\ra^{1-s}  \prod_{j=0}^3 f(\xi_j,\tau_j) }{ \la \xi_1\ra^s \la \xi_3\ra^{ s}     \la\tau_0+\xi_0^2\ra^{\frac12-}\prod_{j=1}^3\la\tau_j+\xi_j^2\ra^{\frac12+}}\les \|f\|_{L^2_{\xi,\tau}}^4.
$$
Assume that $\la\tau_0+\xi_0^2\ra =\max\limits_{j=0,\dots,3} \la\tau_j+\xi_j^2\ra$, the other cases are similar. This  implies that 
$$
\la\tau_0+\xi_0^2\ra\gtrsim \la (\xi_0-\xi_1)(\xi_0-\xi_3)\ra =\la (\xi_0-\xi_1)(\xi_1-\xi_2)\ra .
$$ 
Using $\la\xi_0\ra\approx \la \xi_1\ra$, $\la\xi_3\ra\approx \la \xi_2\ra$, and letting $\rho=\xi_0-\xi_1$ in the $\xi_0$ integral, we bound the left hand side by
$$ 
\int\limits_{    |\rho |<1   } 
 \frac{\la\xi_1\ra^{ a} \la \xi_2\ra^{1-2s}   f(\rho+\xi_1,\tau_1-\tau_2+\tau_3)f( \xi_1,\tau_1)f( \xi_2,\tau_2) f( \rho+\xi_2,\tau_3)}{ \la \rho (\xi_1-\xi_2)\ra^{\frac12-}  \la\tau_1+\xi_1^2\ra^{\frac12+}\la\tau_2+\xi_2^2\ra^{\frac12+} \la\tau_3+(\rho+\xi_2)^2\ra^{\frac12+}}d\rho d\xi_1 d\xi_2 d\tau_1 d\tau_2d\tau_3.
$$
Noting that  $ \frac{\la\xi_1\ra^{ a} \la \xi_2\ra^{1-2s}  }{ \la \rho (\xi_1-\xi_2)\ra^{\frac12-} }\les \rho^{-\frac12+}$ for $a<\min(2s-1,\frac12)$ and $|\rho|<1$, and integrating in $\rho$, we obtain the bound
$$ 
\sup_\rho \int 
 \frac{    f(\rho+\xi_1,\tau_1-\tau_2+\tau_3)f( \xi_1,\tau_1)f( \xi_2,\tau_2) f( \rho+\xi_2,\tau_3)}{     \la\tau_2+\xi_2^2\ra^{\frac12+} \la\tau_3+(\rho+\xi_2)^2\ra^{\frac12+}} d\xi_1 d\xi_2 d\tau_1 d\tau_2d\tau_3.
$$
By the Cauchy-Schwarz inequality and Fubini's theorem,  the integral above is bounded by the square root of 
\begin{multline*}
\int 
 \frac{    f^2(\rho+\xi_1,\tau_1-\tau_2+\tau_3) f^2( \xi_2,\tau_2)  }{  \la\tau_3+(\rho+\xi_2)^2\ra^{1+}}d\xi_1 d\tau_1 d\tau_3 d\xi_2  d\tau_2 \times \\ \int 
 \frac{    f^2( \xi_1,\tau_1)  f^2( \rho+\xi_2,\tau_3)}{     \la\tau_2+\xi_2^2\ra^{1+} }d\xi_1 d\tau_1 d\tau_2 d\xi_2  d\tau_3\les \|f\|^{8}_{L^2_{\xi,\tau}}.
\end{multline*}
We now consider $NR_1$.  By duality and renaming the functions, it suffices to prove that
$$
 \int\limits_{\xi_0-\xi_1+\xi_2-\xi_3=0  \atop {\tau_0 - \tau_1+\tau_2-\tau_3  =0}} \frac{\la \xi_0\ra^{s+a} \la\xi_2\ra^{1-s}  \la\tau_3+\xi_3^2\ra^{\frac38 }\prod_{j=0}^3 f(\xi_j,\tau_j) }{ \la \xi_0-\xi_1\ra \la \xi_0-\xi_3\ra \la \xi_1\ra^s\la  \xi_3\ra^s \la\tau_0+\xi_0^2\ra^{\frac12- }\la\tau_1+\xi_1^2\ra^{\frac12+ } \la\tau_2+\xi_2^2\ra^{\frac12+ }}
 \les \|f\|^4_{L^2_{\xi,\tau}}.
$$
We will consider the following cases:
 $$\la\tau_3+\xi_3^2\ra\les \la \tau_0+\xi_0^2\ra \,\,\,\text{ and }\,\,\, \max\limits_{j=0,\dots,3} \la\tau_j+\xi_j^2\ra\les 
\la \xi_0-\xi_1\ra \la \xi_0-\xi_3\ra.$$  The remaining cases are similar.  In the former case, by the Cauchy-Schwarz inequality we bound the left hand side by the square root of 
\begin{multline*}
 \int\limits_{\xi_0-\xi_1+\xi_2-\xi_3=0  \atop {\tau_0 - \tau_1+\tau_2-\tau_3  =0}}\frac{ M^2f^2(\xi_0,\tau_0)}{ \la\tau_1+\xi_1^2\ra^{ 1 + } \la\tau_2+\xi_2^2\ra^{ 1 + }} \times  \int\limits_{\xi_0-\xi_1+\xi_2-\xi_3=0  \atop {\tau_0 - \tau_1+\tau_2-\tau_3  =0}}   \prod_{j=1}^3 f^2(\xi_j,\tau_j)  \\ \les \|f\|_{L^2_{\xi,\tau}}^8\sup_{\xi_0} \int\limits_{\xi_0-\xi_1+\xi_2-\xi_3=0  }  M^2,
\end{multline*}
where $M=\frac{\la \xi_0\ra^{s+a} \la\xi_2\ra^{1-s}   }{ \la \xi_0-\xi_1\ra \la \xi_0-\xi_3\ra \la \xi_1\ra^s\la  \xi_3\ra^s  }$. The supremum is finite as it was considered for the $B$ term above.

In the latter case when $\la \xi_0-\xi_1\ra \la \xi_0-\xi_3\ra  \gtrsim  \max\limits_{j=0,\dots,3} \la\tau_j+\xi_j^2\ra, $  we have
\begin{multline*}
\frac{ \la\tau_3+\xi_3^2\ra^{\frac38 }  }{ \la \xi_0-\xi_1\ra \la \xi_0-\xi_3\ra   \la\tau_0+\xi_0^2\ra^{\frac12- }\la\tau_1+\xi_1^2\ra^{\frac12+ } \la\tau_2+\xi_2^2\ra^{\frac12+ }} \\ \les \frac{ 1  }{ \la \xi_0-\xi_1\ra^{\frac58-} \la \xi_0-\xi_3\ra^{\frac58-}   \la\tau_0+\xi_0^2\ra^{\frac12+ }\la\tau_1+\xi_1^2\ra^{\frac12+ } \la\tau_2+\xi_2^2\ra^{\frac12+ }}. 
\end{multline*}
Using and \eqref{ks}, \eqref{mf}, and defining $\xi_{mid}$ and $\xi_{max}$ as in the proof of Proposition~\ref{prop:smooth3}, it suffices to prove that 
$$
\sup\limits_{\xi_0-\xi_1+\xi_2-\xi_3=0 }  \frac{\la \xi_0\ra^{s+a} \la\xi_2\ra^{1-s}   }{ \la \xi_0-\xi_1\ra^{\frac58-} \la \xi_0-\xi_3\ra^{\frac58-} \la \xi_1\ra^s\la  \xi_3\ra^s }\,\frac{ \la \xi_{mid} \ra^{\frac14}}{ \la \xi_{max}\ra^{\frac14}} 
 <\infty.
$$
This bound follows from the bound for  \eqref{temp43} above in all cases except $|\xi_1|\leq |\xi_3|\les |\xi_2|\approx |\xi_0|$.  In this case, we bound the supremum by
\be \label{temp49}
\sup\limits_{{\xi_0-\xi_1+\xi_2-\xi_3=0 }\atop{|\xi_2|\approx |\xi_0|}}  \frac{\la \xi_0\ra^{\frac34+a}   }{ \la \xi_0-\xi_1\ra^{\frac58-} \la \xi_0-\xi_3\ra^{\frac58-} \la \xi_1\ra^s\la  \xi_3\ra^{s-\frac14} }.
\ee
For $s\geq \frac78$, we bound \eqref{temp49}  by $\la\xi_0\ra^{a-\frac12+}$, which is bounded provided that $a<\frac12$. For $\frac58\leq s<\frac78$, we have the bound
$$
\eqref{temp49} \les \sup\limits_{{\xi_0-\xi_1+\xi_2-\xi_3=0 }\atop{|\xi_2|\approx |\xi_0|}}  \frac{\la \xi_0\ra^{\frac34+a}   }{ \la \xi_0 \ra^{\frac58-} \la \xi_0 \ra^{s-\frac14} \la \xi_1-\xi_2\ra^{\frac78-s-} \la \xi_1\ra^{s-\frac58}  } \les \sup_{\xi_0}\la \xi_0\ra^{\frac38+a-s-\min(\frac78-s,s-\frac58)+},
$$
which is finite provided that $a<\min(2s-1,\frac12)$. For $\frac12< s<\frac58$, we have the bound
$$
\eqref{temp49}\les  \sup_{\xi_0} \la \xi_0\ra^{ a-2s+1+}  <\infty  
$$
 provided that $a<2s-1.$ The proof for $NR_2$ is identical to the proof for $NR_1$. 
\end{proof}
The following lemma provides a bound for $\|w\|_{X^{s,-\frac38}}$ that arise in the estimate for the terms   $NR_1(u)$ and $NR_2(u)$ in Proposition~\ref{prop:smooth5}.
\begin{lemma}\label{lem:b38} For $s>\frac12$ we have
$$
\|w\|_{X^{s,-\frac38}}\les \|u\|_{X^{s,\frac12+}}^3+\|u\|_{X^{s,\frac12+}}^5.
$$
\end{lemma}
\begin{proof} We will only provide the bound for the cubic part since the quintic part follows easily from Sobolev embedding. As in the proof of Proposition~\ref{prop:smooth3}, it suffices to prove that
$$
\sup\limits_{\xi_0 - \xi_1 + \xi_2-\xi_3=0 }\frac{\la\xi_0\ra^{s } \la \xi_1\ra^{-s}\la \xi_2\ra^{1-s}  \la \xi_3\ra^{-s}   }{ \la (\xi_0-\xi_1)(\xi_0-\xi_3)\ra^{\frac38}   } \frac{\la \xi_{mid}\ra^{\frac14}}{\la \xi_{max}\ra^\frac14}<\infty. 
$$
The case  $|\xi_0-\xi_1|\les 1$ or   $|\xi_0-\xi_3|\les 1$ is immediate. Thus it is enough to prove that
\be\label{temp45}
\frac{\la\xi_0\ra^{s } \la \xi_1\ra^{-s}\la \xi_2\ra^{1-s}  \la \xi_3\ra^{-s}   }{ \la  \xi_0-\xi_1 \ra^{\frac38}  \la  \xi_0-\xi_3 \ra^{\frac38}}  \frac{\la \xi_{mid}\ra^{\frac14}}{\la \xi_{max}\ra^\frac14}
\ee
is bounded when $\xi_0 - \xi_1 + \xi_2-\xi_3=0$ and  $|\xi_1|\leq |\xi_3|$, by symmetry. 

In the case  $|\xi_1|\leq |\xi_3|\les |\xi_2|\approx |\xi_0|$, for $s\geq \frac58$, we have  
$$
\eqref{temp45}\les \frac{\la\xi_0\ra^{\frac34  }      }{  \la \xi_1\ra^{ s}\la  \xi_0-\xi_1 \ra^{\frac38}  \la  \xi_0-\xi_3 \ra^{\frac38}  \la \xi_3\ra^{ s-\frac14}}  \les 1,$$
 whereas, for $\frac12<s<\frac58$, we find 
 $$ \eqref{temp45}\les \frac{\la\xi_0\ra^{\frac34 }      }{  \la \xi_1\ra^{ s-\frac38}\la  \xi_0 \ra^{\frac38}  \la  \xi_1-\xi_2 \ra^{\frac58-s }  \la \xi_0\ra^{ s-\frac14}}\les   \frac{\la\xi_0\ra^{\frac58-s }      }{  \la \xi_2\ra^{\frac58-s}  }\les 1.$$  
In the case  $|\xi_1|\leq |\xi_3| \approx |\xi_0|$ and $|\xi_3|\gg|\xi_2|$, we have 
$$\eqref{temp45}\les \frac{  \la \xi_1\ra^{-s}\la \xi_2\ra^{1-s}   }{ \la  \xi_3  \ra^{\frac38}  \la  \xi_1-\xi_2 \ra^{\frac38}} \les   
\frac{  \la \xi_2\ra^{1-s}    }{ \la  \xi_3  \ra^{\frac38}  \la   \xi_2 \ra^{\frac38}}\les 1.  
$$ 
In the case  $|\xi_1|\leq |\xi_3| \approx |\xi_2|$ and $|\xi_3|\gg|\xi_0|$, we have 
$$
\eqref{temp45}\les \frac{\la\xi_0\ra^{s } \la \xi_1\ra^{-s}  \la \xi_3\ra^{\frac38-2s }   (\la \xi_0\ra^{\frac14} +\la \xi_1\ra^{\frac14}) }{ \la  \xi_0-\xi_1\ra^{\frac38}   } \les  
 \la\xi_0\ra^{s-\frac38 }  \la \xi_3\ra^{\frac38-2s }   (\la \xi_0\ra^{\frac14} +\la \xi_1\ra^{\frac14})  
\les 1.
$$ 
In the case  $|\xi_0|,|\xi_2|\ll |\xi_1| \approx |\xi_3|$, we have 
$$
 \eqref{temp45}\les \la\xi_0\ra^{s }  \la \xi_2\ra^{1-s}  \la \xi_3\ra^{-2 s-1} ( \la  \xi_0\ra^{\frac14}+\la  \xi_2\ra^{\frac14})     \les 1.  
$$  
This completes the proof of the lemma.
\end{proof}
\begin{proof}[Proof of Theorem~\ref{thm:smoothR}]
 Integrating  \eqref{preduhamelR} on $[0,t]$ we obtain
$$
 u(t)-e^{it\Delta}g =B(u(t)) -e^{it\Delta}B(g) + i \int_0^te^{i(t-s) \Delta } \big(R(u) +\frac12|u|^4u+ NR_1(u) +NR_2(u)  \big)ds.
$$
The claim follows from this using  the bounds in Proposition~\ref{prop:smooth5}, the inequality \eqref{eq:xs2}, and the embedding $X^{s,b}\subset C^0_t H^{s}_x$ for $b>\frac{1}{2}$.
 \end{proof}

\section{Appendix}\label{sec:appendix} 
In this Appendix we discuss two lemmas. For a proof of the first one, see \cite{et3}. The second one is from \cite{ckstt} but we provide a proof for  completeness.
\begin{lemma}\label{lem:sums} If $\beta \geq \gamma \geq 0$ and $\beta + \gamma > 1$, then
\[ \int \frac{1}{\langle x-a_1 \rangle ^\beta \langle x - a_2 \rangle^\gamma} dx \lesssim \langle a_1-a_2 \rangle ^{-\gamma} \phi_\beta(a_1-a_2),  \]
where
\[ \phi_\beta(a) \sim \begin{cases} 1 & \beta > 1 \\  \log(1 + \langle a \rangle ) & \beta = 1 \\ \langle a \rangle ^{1-\beta} &\beta < 1 .\end{cases} \]
\end{lemma}

\begin{lemma}\label{lem:contgauge}
For $s>\frac12$, the gauge $\mG_\alpha f(x)= f(x)\exp\big( -i\alpha\int_x^\infty |f(y)|^2dy\big)$ is Lipschitz continuous on bounded subsets of $H^s$. 
\end{lemma}
\begin{proof} Let $\|g\|_{Lip}:=\|g\|_{L^\infty}+\|g^\prime\|_{L^\infty}$.
Sobolev embedding gives
$$\Big\|\exp\big( -i\alpha\int_x^\infty |f(y)|^2dy\big)\Big\|_{Lip} \les 1+ \|f\|_{H^s}^2,\,\,\,s>\frac12.$$
In addition, by interpolation between $L^2$ and $H^1$, we have
$$
\|fg\|_{H^s}\les \|f\|_{H^s}\|g\|_{Lip},\,\,\,s\in [0,1],
$$
which completes the proof for $s\in(\frac12,1]$. 

For $s\in(1,2]$, let $g=\exp\big( -i\alpha\int_x^\infty |f(y)|^2dy\big)$, and note that
\begin{multline*}
\|fg\|_{H^s}\les \|fg\|_{L^2}+\|f^\prime g\|_{H^{s-1}}+\|f^3 g \|_{H^{s-1}} \\ \les 
\|g\|_{Lip}\big(\|f\|_{L^2}+\|f^\prime  \|_{H^{s-1}}+\|f^3  \|_{H^{s-1}} \big) \les 1+ \|f\|_{H^s}^5.
\end{multline*}
For $s>2$ the same argument works inductively. 
\end{proof}

\end{document}